\documentclass[reqno]{amsart}
\usepackage{graphicx}
\usepackage{amsfonts,amsmath,amssymb}
\newcommand{\rl}{{\mathbb{R}}}
\newcommand{\cx}{{\mathbb{C}}}
\newcommand{\id}{{\mathbb{I}}}

\newcommand{\dbar}{\overline{\partial}}

\newcommand{\abs}[1]{\left|{#1}\right|}
\newcommand{\norm}[1]{\left\|{#1}\right\|}
\newcommand{\e}{\epsilon}
\newcommand{\supp}{{\mathrm{supp}}}
\newcommand{\tmop}[1]{\ensuremath{\operatorname{#1}}}
\renewcommand{\Re}{\tmop{Re}}

\newcommand{\dist}{{\mathrm{dist}}}

\newcommand{\CR}{{\mathrm{CR}}}
\newtheorem{theorem}{Theorem}
\newtheorem{lemma}{Lemma}
\newtheorem{prop}{Proposition}

\newtheorem*{thma}{Result  (\cite[Theroem~3.3]{stasica},\cite[Proposition~1.4]{kr})}
\newtheorem*{thmb}{Result  (\cite[Chapter VI, Theorem 2]{stein})}

\theoremstyle{definition}
\newtheorem{defn}{Definition}
\theoremstyle{remark}

\theoremstyle{remark}

\newcommand{\reg}[1]{{{#1}^{\mathrm{reg}}}}
\newcommand{\sng}[1]{{{#1}^{\mathrm{sing}}}}
\newcommand{\esup}{\mathop{\mathrm{ess. sup}_{\mathcal{H}}}}
\newenvironment{myprop}[1]{\mbox{}\newline{\bf{#1}.}\em}{\mbox{}\\}

\begin{document}
\title{ CR functions on subanalytic hypersurfaces.}
\author{Debraj Chakrabarti}
\address{Department of Mathematics,  University of Notre Dame,
Notre Dame, IN 46556, USA} \email{dchakrab@nd.edu}
\author{Rasul Shafikov}
\address{Department of Mathematics, the University of Western Ontario,
London, Canada  N6A 5B7} \email{shafikov@uwo.ca}
\begin{abstract}
A general class of singular real hypersurfaces, called {\em subanalytic}, is defined.
For a subanalytic hypersurface $M$ in $\mathbb C^n$, Cauchy-Riemann (or simply CR) functions 
on $M$ are defined, and certain properties of CR functions discussed. In particular,
sufficient geometric conditions are given for a point $p$ on a subanalytic 
hypersurface $M$  to admit a germ at $p$ of a smooth CR function $f$ that cannot be 
holomorphically extended to either side of $M$. As a consequence it is shown that a well-known 
condition of the absence of complex hypersurfaces contained in a smooth real hypersurface $M$, 
which guarantees one-sided holomorphic extension of CR functions on $M$, is neither a necessary 
nor a sufficient condition for one-sided holomorphic extension in the singular case.
\end{abstract}
\maketitle

\setcounter{tocdepth}{1}
\tableofcontents

\section{Main Results}

In this paper we define a class of non-smooth real hypersurfaces in
$\cx^n$, which we call {\em subanalytic}, and study general
properties of CR functions defined on them. Precise definitions are
given in $\S$\ref{sec-subanalytic-def} and $\S$\ref{sec-CR}, but
roughly speaking, a subanalytic hypersurface $M\subset\cx^n$ is a
real codimension one subanalytic set which divides any small enough
neighbourhood $U$ of any point $p\in M$ into two one-sided
neighbourhoods~$U^\pm$. This gives a well-defined local orientation
on the smooth part $\reg{M}$ of $M$, which allows us to do
integration. We define  a locally integrable function $f$ on a
subanalytic hypersurface $M$ to be CR if it satisfies
$\int_{\reg{M}}f\dbar\phi=0$, where $\phi$ is a test form of
bidegree $(n,n-2)$. In Proposition~\ref{prop-bv}, we show that, just
as with smooth hypersurfaces, continuous boundary values of
holomorphic functions on subanalytic hypersurfaces are CR.

It follows from the definition that the restriction of a CR function
on a subanalytic set $M$ to its regular part $\reg{M}$ is CR in the
usual sense. Although a function on $M$ which is CR on $M^{\rm reg}$
may fail in general to be CR on $M$, our first theorem gives some
sufficient conditions for a CR function on the smooth locus of a
hypersurface $M$ to be CR on all of~$M$.

\begin{theorem}\label{thm-removable}
Let $M$ be a subanalytic hypersurface in $\cx^n$, and let $E$ be
subanalytic subset of $\cx^n$ contained in $M$ such that $E$ is
of codimension at least one in $M$. Suppose that $f$ is a function
on $M$ which is CR on $M\setminus E$.
\begin{enumerate}
\item[(i)] If $f$ is continuous on $M$ and vanishes on $E$, then $f$ is CR on $M$.
 \item[(ii)] If $E$ has codimension at least two in $M$, and for $z$ near $E$ we have
    $\abs{f(z)}=O(\dist(z,E)^{-\alpha})$, where $0\leq \alpha<1$, then $f$ is CR on $M$.
    In particular, if $f$ is bounded on $M$, then $f$ is a CR function on~$M$.
\end{enumerate}
\end{theorem}

Theorem~\ref{thm-removable} is proved in Section~\ref{sec-removable}
using the classical technique of Bochner \cite{b}. This method can
be  used to prove more general results of this type, but we only
develop the topic to the extent needed for our purposes.

The main thrust of this paper is in determining the conditions under
which there exist continuous CR functions on subanalytic
hypersurfaces which are not (locally) boundary values of holomorphic
functions. For a smooth hypersurface $M$ in $\cx^n$, the existence
of such functions on  $M$ near a point $p\in M$ is equivalent to the
existence of the germ of a complex analytic hypersurface $A$ through
$p$ contained in $M$.  We say that a real hypersurface $M$ (smooth
or subanalytic) in a complex manifold is {\em non-minimal} at a
point $p\in M$, if there is a germ of a complex analytic
hypersurface $A$ such that $p\in A$ and $A\subset M$. If there is no
such germ $A$, we say that $M$ is {\em minimal} at $p$. Therefore,
non-extendable CR functions exist near a point on a smooth
hypersurface (or even a hypersurface that can be locally represented
as a graph near a point) provided the hypersurface is non-minimal at
that point.

For subanalytic hypersurfaces, the condition for the existence of
non-extendable CR functions is more subtle. First of all, without an
additional topological assumption, non-minimality by itself does not
suffice. Further, there is another geometric condition, first
introduced in \cite{cs}, called {\it two-sided support}, that also
gives rise to non-extendable CR functions. We say that $M$ has {\it
proper} two-sided support at $p\in M$ if there is an open
neighbourhood $\Omega$ of $p$ such that $M$ divides $\Omega$ into
two connected components $\Omega^+$ and $\Omega^-$, and there exist
germs at $p$ of distinct complex analytic hypersurfaces
$A^\pm\subset\overline{\Omega^\pm}$ such that $A^+\cap M = A^-\cap
M$, see $\S$\ref{sec-twosided} for details. We have the following
sufficient conditions for the existence of smooth non-extendable CR
functions on subanalytic hypersurfaces.

\begin{theorem}\label{thm-non-ext}
Let $M$ be a subanalytic hypersurface in $\cx^n$ and $p\in M$.
Suppose that at least one of the following statements holds:
\begin{enumerate}
\item $M$ is non-minimal at $p$, i.e., there exists a germ
    of a complex hypersurface $Z\subset M$ such that $p\in
    Z$, and $Z$ divides $M$ locally into more than
    one component at $p$.
\item $M$ has proper two-sided support at $p$.
\end{enumerate}
Then, for every integer $m\geq 0$ there is a CR function on $M$ near
$p$ of class $\mathcal{C}^m$ that does not extend as a holomorphic
function to either side.
\end{theorem}

By a $\mathcal{C}^m$-smooth function on $M$, we mean a function
which is the restriction to $M$ of a $\mathcal C^m$-smooth function
on $\cx^n$. It should be noted that the assumption that the
hypersurface $M$ is subanalytic plays a crucial role in the proof.
While the notion of non-minimality and proper two-sided support
makes sense for singular hypersurfaces of more general types, the
proof of Theorem~\ref{thm-non-ext} would fail for more general
singular hypersurfaces when $m\geq 1$.

Note that if a hypersurface $M$ can be represented as a graph near a
point $p$, and $M$ is non-minimal at $p$, then the complex
hypersurface contained in $M$ must divide $M$ into two components
(see \cite{chirka:rado}). Hence, the assumption that $Z$ locally
divides $M$ in (1) of Theorem~\ref{thm-non-ext} is automatically
satisfied. However, in general, a subanalytic hypersurface cannot be
locally represented as the graph of a function. At such a non-graph
point $p\in M$, it is possible that $M$ is non-minimal, but the
complex hypersurface $A\subset M$ through $p$ does not locally
divide $M$ near $p$. In Sections~\ref{sec-m} and~\ref{sec-n} we
consider some examples of such hypersurfaces, in particular, we
prove the following.

\begin{theorem}\label{mainthm}\mbox{}
\begin{enumerate}
\item[(i)] For $n\geq 2$, let
    \[M_1=\{(z_1,\ldots,z_n)\in\mathbb C^n:
    \Re(z_1z_2)+|z_1|^2=0\}.\] Then $M_1$ is non-minimal at the origin,
    but if $f$ is a bounded CR function near the origin,
    then $f$ extends holomorphically to one side of $M_1$. Further, if
    $f$  is $\mathcal{C}^k$-smooth
    on $M$ near 0, then the extension is $\mathcal{C}^k$ up to the boundary.

\item[(ii)] For $n\ge 3$, let \[ N_1=\{(z_1,\ldots,z_n)\in \mathbb C^n: \Re(z_1 z_2 + z_1\overline
z_3)=0\}.\]
    Then $N_1$ is non-minimal at the origin, but every bounded CR function  on $N_1$ near the origin
    extends holomorphically to a full neighbourhood of the origin.
\end{enumerate}
Further, there are infinitely many biholomorphically inequivalent hypersurfaces $M\subset\cx^n$,
$N\subset\cx^n$ with  properties described in (i) and (ii).
\end{theorem}

Consequently, {\em minimality is not a necessary condition for the local one-sided holomorphic extension of
all CR functions on a singular hypersurface.} This is in striking contrast with smooth or graph-like
hypersurfaces.

In part (i), when $f$ is merely bounded, the boundary value on $M_1$
of the extension should be understood in the sense of distributions
on the smooth part, see $\S$\ref{sec-bv} for details. Part (ii)
implies that every bounded CR function on $N_1$ near 0 is actually
real-analytic. The proof of Theorem~\ref{mainthm} is based on the
construction of the local envelope of holomorphy of an arbitrarily
thin one-sided neighbourhood of $M\setminus\{z_1=0\}$ (resp.
$N\setminus\{z_1=0\}$), followed by an application of the Lewy
extension theorem. The envelopes are obtained by an explicit
construction of analytic discs and the Kontinuit\"atssatz.

In the smooth (or even graph-like) situation, minimality is a
necessary, as well as sufficient condition for local
holomorphic extension of CR functions to one side. The
sufficiency of minimality for $\mathcal{C}^2$ hypersurfaces is
a celebrated result \cite{tu} of Tr\'{e}preau. This has been
generalized to graphs of continuous real-valued functions by
Chirka \cite{chirka:graphs}. It was shown in \cite{cs} that for
singular hypersurfaces, minimality is no longer a sufficient
condition for local one-sided holomorphic extension (see $\S$\ref{sec-twosided}).
Combining this with Theorem~\ref{mainthm} above we conclude that {\it for singular
hypersurfaces minimality is neither a necessary nor a
sufficient condition for one-sided holomorphic extension of CR
functions.}

{\em Acknowledgments.} The first author would like to thank
J.-P. Rosay, E. Bedford and M.-C. Shaw for helpful discussions, and
the organizers of the SCV semester at Institut Mittag-Leffler,
Djursholm for an invitation to their program, where part of this
work was done; very special thanks are due to Meera and Tanmay for
all their support. Research of the second author is supported in part 
by the Natural Sciences and Engineering Research Council of Canada.

\section{Subanalytic sets and hypersurfaces}\label{sec-subanalytic}

\subsection{Definitions}\label{sec-subanalytic-def}
Recall that a subset $E$ of a real analytic manifold $X$ is called {\em semianalytic} if it
is locally defined by finitely many real analytic equations and inequalities. More precisely,
for each $p\in X$, there is a neighbourhood $U$ of $p$, and real analytic in $U$
functions $f_i,g_{ij}$, where $i=1,\ldots, r$, $j=1,\ldots, s$, such that
\[ E\cap U =\bigcup_{i=1}^r\left(\bigcap_{j=1}^s\{x\in U\colon
g_{ij}(x)>0 \text{ and } f_i(x)=0\}\right).\] A real analytic set is
clearly semianalytic. A {\em subanalytic} subset $E$ of a real
analytic manifold $X$  is one which can be locally represented as
the projection of a semianalytic set. More precisely, for every $p\in X$, there exist a neighbourhood
$U$ of $p$ in $X$, a real analytic manifold $Y$, and a relatively compact semianalytic set
$Z\subset X\times Y$ such that $E\cap U=\pi(Z)$, where $\pi:X\times
Y\rightarrow X$ is the natural projection. In particular, semianalytic sets are subanalytic.
An excellent reference on semi- and subanalytic sets is \cite{bm}.

The {\it dimension} of a subanalytic set $E$, $\dim E$, is the maximal dimension of the germ of a real
analytic submanifold contained in $E$. If $E$ is a subanalytic subset of a manifold $X$, by $\reg{E}$ (the
regular points of $E$) we denote the set of points $p\in E$ near which $E$ is a (real-analytic) submanifold of 
$X$ of dimension $\dim E$. Its complement $E\setminus\reg{E}$ (the singular locus) is denoted by 
$\sng{E}$. 

We now define the class of objects in which we are interested. Let
$\mathcal{X}$ be a topological space, let $\mathcal{Y}\subset
\mathcal{X}$ be a locally closed subset (i.e., for each
$q\in\mathcal{Y}$ there is a neighbourhood $V$ of $q$ in
$\mathcal{X}$ such that $V\cap\mathcal{Y}$ is a closed subset of
$V$), and let $p\in\mathcal{Y}$. We will say that $\mathcal{Y}$ {\em
locally separates $\mathcal{X}$ at $p$} if the following holds: for
every neighbourhood $V$ of $p$ in $\mathcal{X}$, there is a {\em connected}
open neighbourhood $U$ of $p$ contained in $V$, such that the set
$U\setminus\mathcal{Y}$ has exactly two open connected components
$U^\pm$, and we have $\overline{U^+}\cap\overline{U^-}=\mathcal{Y}\cap\overline{U}$.
A smooth hypersurface in $\rl^N$ locally separates $\rl^N$ at each point, which is part of our intuitive
idea of a hypersurface. However, if $\mathcal{X}$ is a manifold, the notion of being locally separating
is much weaker than being a topological submanifold . For example, the set
$\mathcal{Y}=\{x\in\rl^4\colon x_1^2+x_2^2-x_3^3-x_4^2=0\}$ locally separates $\rl^4$ at each of
its points, but is not a topological submanifold of $\rl^4$ near 0.

We can now make the following definition.
\begin{defn}\label{def:sub}
A locally closed subanalytic subset $M$ of a real analytic manifold
$X$ is called a {\em subanalytic hypersurface} in $X$ if $M$ locally
separates $X$ at each point. \end{defn}

We note some elementary properties of subanalytic hypersurfaces.
\newcounter{property}
\begin{list}{\arabic{property}$^\circ$}{\leftmargin=0cm \itemindent=1cm}
\usecounter{property}
\item  Let $S$ be the set of smooth points of $M$, i.e., the set of points near which $M$ is a smooth manifold.
   Then $S$  is a manifold (possibly non-closed and possibly non-connected) of codimension one in $X$.  Indeed, 
   $S$ is a submanifold of $X$ near each of its points. But no submanifold of codimension two or more can locally 
   separate $X$.

\item Let $\Omega$ be a neighbourhood of $p\in M$ such that
    $\Omega\setminus M$ has two connected components
    $\Omega^\pm$. Then each component of
    $\Omega\cap\reg{M}$ is orientable, and it is possible to assign
    an orientation to each component in such a way that the positive normal points into
    $\Omega^+$ at each point.

To see this, consider the function $\rho$ on $X$ given
\[ \rho(x)=\begin{cases}
    \dist(x,M) &\text{ for $x\in\Omega^+$}\\
    -\dist(x,M) &\text{ for $x\in\Omega^-$}
    \end{cases}
    \]
Thanks to \cite[Prop.~7.4]{bm}, the function $\rho^2$ is real
analytic in a neighbourhood of $\reg{M}$. It is easy to see that
$\rho$ itself is smooth in a neighbourhood of $\reg{M}$ and
$\nabla\rho$ is non-zero on $\reg{M}$. We orient each component of
$\reg{M}$ so that $\nabla\rho$ is the positive normal.

\item Here and in the sequel we denote by $\mathcal{H}$ the Hausdorff
 measure of codimension $1$ in $\rl^N$. Let $M\subset \mathbb R^N$ be a subanalytic
hypersurface. We denote by $L^1_{\rm loc}(M)$ the space of functions
on $M$ that are locally integrable on $M$ with respect to the
measure $\mathcal{H}$ in the following sense: $f\in L^1_{\rm
loc}(M)$, if for every compact $K$ contained in an open set
$U\subset X$, where $U$  has the property that $M\cap U$ is closed
in $U$, we have $\int_{K\cap M}\abs{f}d\mathcal{H}<\infty$ (since
$M$ is locally closed, every point has a neighbourhood in which $M$
is closed). It will be clear from Lemma~\ref{lem-volumes} that every
bounded function on $M$ is in $L^1_{\rm loc}(M)$.
\end{list}

\subsection{Properties}
Subanalytic sets enjoy several closure properties: locally finite
unions and intersections, set-theoretic differences, complements,
topological closures and interiors, and proper projections onto linear
subspaces of subanalytic sets are subanalytic. The sets of regular
points and the singular locus of a subanalytic set are themselves
subanalytic.

A fundamental property of subanalytic sets \cite{ds} is that they
admit a stratification by real analytic manifolds. In fact, any
subanalytic set $X$ in an open set $\Omega\subset\rl^N$ can be
represented as a locally finite disjoint union of subanalytic
subsets $A_i$ of $\Omega$, where each $A_i$ is a (possibly
non-closed) real analytic submanifold of $\Omega$, and the family
$\{A_i\}$ satisfies the {\it frontier} condition: if
$A_k\cap\overline{A_l}$ is nonempty, then $A_k\subset
\overline{A_l}$, and $\dim A_k<\dim A_l$. For a subanalytic
hypersurface, which is a bounded subanalytic subset of $\rl^N$, the
stratification will consist of a finite number of strata. Moreover,
given a subanalytic set $E\subset X$, the stratification of $X$  may
be chosen to be compatible with $E$, i.e., $E$ (and therefore
$X\setminus E$) is a union of strata. The maximum dimension of a stratum in
a stratification of a subanalytic set $E$ equals $\dim E$, the dimension of
$E$, and therefore, it is independent of the choice of stratification.

An important property of subanalytic sets is {\em \L
ojasiewicz's inequality.} In a simple form that suffices for
our purposes it can be stated as follows: let $K$ be a subset
of $\rl^N$, and $f: K\to \mathbb \rl$ be a function such
that its graph is a compact subanalytic set in $ \rl^{N+1}$,
and let $X=f^{-1}(0)$. Then  there exist $C,r>0$ such that for
any $x\in K$,
\begin{equation}\label{eq:loj-ineq}
|f(x)|\ge C\,{\rm dist}(x,X)^r.
\end{equation}
\L ojasiewicz's inequality implies that any two subanalytic
sets $X$ and $Y$ in $\rl^N$ are {\it regularly situated}, i.e.,
for any $x_0\in X\cap Y$ there exist a neighbourhood $V$ of
$x_0$, and $C,r>0$ such that for any $x\in V$,
\begin{equation}\label{eq:regsit}
{\rm dist}(x,X)+ {\rm dist}(x,Y) \ge C\, {\rm dist}(x,X\cap Y)^r.
\end{equation}
For more details see, e.g., \cite{bm}. These inequalities are
crucial for our construction of smooth non-extendable CR
functions.

Now we recall some metric properties of subanalytic sets. Let
$\Gamma\subset\rl^N$ be an analytic subvariety (not necessarily
closed.) Suppose there exist an open set $U$ in a linear subspace
$E\subset\rl^N$, with $\dim E=k$, and an analytic map
$\phi:U\rightarrow E^\perp$ with values in the orthogonal complement
of $E$, such that $\Gamma$ is the graph of $\phi$ in $E\bigoplus
E^\perp=\rl^N$. We say that $\Gamma$ is an {\em $\epsilon$-analytic patch}
if the differential of $\phi$ is bounded by $\epsilon$, i.e.,
$\norm{d\phi(x)}<\epsilon$ for each $x\in U$.

Note that a smooth analytic submanifold of $\rl^N$ can be covered by
a family of $\epsilon$-analytic patches, thanks to the implicit
function theorem. The following result  shows that any bounded
subanalytic set can be {\em almost} covered by {\em finitely many}
$\epsilon$-analytic patches.

\begin{thma} Let $Y\subset\rl^N$ be a bounded
subanalytic set with $\dim Y=k$, and let $\epsilon>0$. Then there
are disjoint subanalytic sets
$\Gamma_1^\epsilon,\ldots,\Gamma_K^\epsilon$ contained in $Y$ such
that\begin{enumerate} \item[(i)] $\dim\left(Y\setminus \cup_{i=1}^K
\Gamma_i^\epsilon\right)<k$, and \item[(ii)] the $\Gamma_i^\epsilon$
are $\epsilon$-analytic patches.
\end{enumerate}
\end{thma}

We now draw a corollary from this result. Since we are
interested mainly in hypersurfaces, we assume that the
subanalytic sets are of codimension one, although analogous results
are easily proved for higher codimension.

\begin{lemma}\label{lem-volumes}
Let $M$ be a  bounded subanalytic subset of $\rl^N$, $\dim M = N-1$. There is a
constant $C>0$ such that if  $A$ is an affine subspace of $\rl^N$ with
$A\cap\Omega\subset M$, then
\[
\mathcal{H}(M\cap B(A,r))\leq C r^{N-\dim A-1},
\]
where $B(A,r)=\{x\in\rl^N\colon \dist(x,A)<r\}$.
\end{lemma}

\begin{proof} Let $\mathbb{B}$ be a large enough ball in $\rl^N$
such that $M\subset\mathbb{B}$. Fix $\epsilon>0$, and using the
theorem quoted before this lemma, write ${M}$ as the disjoint union
of a finite number $K$ of $\epsilon$-analytic patches
$\{\Gamma_i\}_{i=1}^K$ each of dimension $N-1$, and a subanalytic
set $R$, with $\dim R< N-1$. Since $\mathcal{H}(R)=0$, it follows
that,
\[\mathcal{H}(B(A,r)\cap {M}) = \sum_{i=1}^K \mathcal{H}(B(A,r)\cap
\Gamma_i).\] Therefore, to prove the claim it is sufficient to
show that for each $i$ we have \begin{equation} \label{eq-piece}
\mathcal{H}(B(A,r)\cap \Gamma_i) \leq  C\cdot r^{N-\dim A-1},
\end{equation} for some constant $C$ independent of $A$. We fix such
an $i$.

There is an $N-1$ dimensional subspace $E_i$ of $\rl^N$, an open set
$U_i\subset E_i$, and a real analytic map $\phi_i:U_i\rightarrow
E_i^\perp$ such that $\Gamma_i$ is the graph of $\phi_i$ in
$E_i\bigoplus E_i^\perp$. Note that $U_i\subset \mathbb B$ automatically holds.
Denote by $\pi_i$ the orthogonal projection
from $\rl^N$ to $E_i$. Then, $\pi_i(\Gamma_i)=U_i$, and
\begin{align*}
\pi_i\left(B(A,r)\cap \Gamma_i\right)&\subset \pi_i(B(A,r))\cap U_i \\
&\subset B(\pi_i(A),r)\cap U_i,
\end{align*}
since $\pi_i$ does not increase distances. Note that $\dim
\pi_i(A)=\dim A$, since $\pi_i$ is injective on $\Gamma_i$. It is
easy to see that
\[ \mathcal{H}\left( B(\pi_i(A),r)\cap U_i\right)\leq C\cdot r^{N-\dim A-1},\]
where  $C$ depends only on the radius of the ball $\mathbb{B}$
but is independent of $A$ and $r$.

Therefore, 
\begin{align*}
\mathcal{H}(B(A,r)\cap \Gamma_i)& =\int_{\pi_i(B(A,r)\cap \Gamma_i)}
 \sqrt{1+ \abs{d \phi_i}^2}d\mathcal{H}(x)\\ &\leq\sqrt{1+\epsilon^2}\cdot
\mathcal{H}(\pi_i(B(A,r)\cap \Gamma_i))\\ &\leq C\sqrt{1+\epsilon^2}r^{N-\dim A-1},
\end{align*}
which proves the result.
\end{proof}

\section{CR functions and removable singularities} \label{sec-removable}

\subsection{CR functions on singular hypersurfaces}\label{sec-CR}
Recall that for a smooth real hypersurface $M$ in a complex manifold
$X$, a function $f$ on $M$ is called {\em CR} if it satisfies the
tangential Cauchy-Riemann equations $\overline {\partial}_b f=0$.
More generally, we can consider distributions which satisfy this
equation. If $f\in L^1_{\rm loc}(M)$ and $M$ is orientable, the
tangential Cauchy-Riemann equations can be rewritten in the adjoint
form: $f$ is CR if and only if for each $\phi\in
\mathcal{D}^{n,n-2}(X)$ (the space of $\mathcal{C}^\infty$ forms
with compact support in $X$ of bidegree $(n,n-2)$) we have
$\int_{M} f \dbar \phi=0$.

We wish to generalize the notion of CR functions to singular
hypersurfaces. In \cite{cs}, we considered the class of
real-analytic hypersurfaces. The smooth part of such
hypersurfaces is always locally orientable, and this allows us
to define CR functions using the adjoint form of the
Cauchy-Riemann equations.

For the more general class of  subanalytic hypersurfaces considered
here, we define CR functions in the same way. Let $M$ be a
subanalytic hypersurface in $\cx^n$, $n\geq 2$, and let $f\in
L^1_{\rm loc}(M)$. We say that $f$ is {\em CR} at $p\in M$, if there
exists an open neighbourhood $\Omega$ of $p$ in $\cx^n$ such that
$\Omega\setminus M$ has exactly two connected components, $\Omega^+$
and $\Omega^-$, and for every $\phi \in
\mathcal{D}^{n,n-2}(\Omega)$, we have
\begin{equation}\label{eq-crdef}
\int_{\reg{M}\cap\Omega} f \dbar \phi=0,
\end{equation}
where each component of $\reg{M}\cap\Omega$ is oriented in such
a way that the positive normal points into $\Omega^+$. We say
that $f$ is CR on $M$ if it is CR at each point of $M$. When
$M$ and $f$ are smooth, this is equivalent to the tangential
Cauchy-Riemann equations $\dbar_b f=0$.

It follows from the definition that being CR is a local
property. If now $\omega$ is an open set of a subanalytic
hypersurface $M$, such that there exists a well-defined
orientation on $\omega$, then a simple argument involving
partition of unity shows that if $f$ is CR at every point of
$\omega$, then $f$ is CR on  $\omega$ in the sense that
\eqref{eq-crdef} holds for some open neighbourhood $\Omega$  of
$\omega$ and all forms $\phi$ with compact support in $\Omega$.

\subsection{Proof of Theorem~\ref{thm-removable}}
Let $M$ be a bounded subanalytic hypersurface in $\mathbb
C^n$. For a subset $E$ of $M$ and a function $f$ on $M$ define \[
S_f(E,r) = \esup_{\frac{r}{2}<\dist(z,E)<r}\abs{f(z)},\] i.e.,
$S_f(E,r)$ is the essential supremum (with respect to the measure
$\mathcal{H}$) of $\abs{f}$ over the points on $M$ which are at
distance at least $\frac{r}{2}$ and at most $r$ from $E$. We let
$CR(M)$ denote the space of CR functions on $M$.

Now let $E$ be a closed subanalytic subset of $M$ of dimension at
most $2n-2$. We fix a stratification of the subanalytic hypersurface
$M$ by finitely many disjoint subanalytic subsets of $\cx^n$ which
are real analytic submanifolds of $\cx^n$ satisfying the frontier
condition and compatible with $E$, i.e., $E$ and $M\setminus E$ are the
union of some strata. We have the following:

\begin{prop}\label{prop-cr} Let $f\in L^1_{\rm loc}(M)\cap \CR(M\setminus E)$,
let $A(p)$ be the stratum through $p\in M$, and let $k(p)=\dim A(p)$. If
at each $p\in E$ we have
\begin{equation}\label{eq-condcr} \lim_{r\rightarrow 0+}
r^{2n-k(p)-2}S_f(A(p),r)=0,\end{equation} then $f$ is CR on $M$.
\end{prop}

\begin{proof}[Proof of Proposition~\ref{prop-cr}]
For the given stratification (compatible with $E$), denote by $E^d$
the union of all strata contained in $E$ of dimension greater than
or equal to $d$. Then $E^0= {E}$, the inclusion $E^{d+1}\subset E^d$
holds, and $E^{2n-1}=\emptyset$. The proof will consist of an inductive process, in
which we assume that $f\in \CR((M\setminus E)\cup E^{d+1})$ and deduce that
$f\in \CR((M\setminus E)\cup E^{d})$. Since by hypothesis,
$f\in \CR((M\setminus E)\cup E^{2n-1})$ the proof will be completed in at
most $2n-1$ iterations of this process.

For $0\leq d\leq 2n-2$, assume $f\in \CR((M\setminus E)\cup E^{d+1})$, and let
$p\in E^d\setminus E^{d+1}$, so that $k(p)=d$. We fix a neighbourhood  $U$ of $p$ in
$\cx^n$, with the following properties : (i) $A(p)\cap U$ can be ``flattened" by a
real analytic diffeomorphism (possible since $A(p)$ is locally a manifold),
(ii) $U$ does not intersect any stratum of dimension less than $d$ (possible by the frontier
condition), (iii) $U\setminus M$ has exactly two components $U^\pm$ (possible since $M$ is
locally separating).

For convenience, let $A=A(p)$, and let $B(A,r)=\{x\in\cx^n\colon \dist(x,A)<r\}$.
It is easy to see that for $r>0$ small there is a cutoff function
$\psi_r$ supported in $B(A,r)$, such that $\psi_r\equiv 1$ in $B(A,\frac{r}{2})$ and
\[ \abs{\dbar\psi_r}= O\left(\frac{1}{r}\right).\]
Note that $\supp(\dbar\psi)\subset B(A,r) \setminus B\left(A,\frac{r}{2}\right)$.

We need to show the following: for every $\phi\in \mathcal{D}^{(n,n-2)}(U)$, we have
$\int_{\reg{M}\cap U} f\dbar\phi=0$, where $\reg{M}\cap U$ is
oriented as in equation~\eqref{eq-crdef}.  We write
$\phi=(1-\psi_r)\phi + \psi_r\phi$. Then $(1-\psi_r)\phi$ is
an $(n,n-2)$ form with support in $U\setminus A$. Since $(U\setminus
A)\cap M\subset (M\setminus E)\cup E^{d+1}$, and  by hypothesis $f$
is CR on $(M\setminus E)\cup E^{d+1}$, we have $\int_{\reg{M}}f
\dbar((1-\psi_r)\phi)=0$. Therefore,
\begin{align*}
\int_{\reg{M}\cap U}f\dbar\phi &= \int_{\reg{M}\cap U}f\dbar(\psi_r\phi)\\
&= \int_{\reg{M}\cap U}f \dbar\psi_r\wedge\phi +\int_{\reg{M}\cap
U}f\psi_r\dbar \phi.
\end{align*}
Since $\psi_r\rightarrow 0$ pointwise on $\reg{M}$ and $f\in
L^1_{\rm loc}(M)$, it follows easily from the dominated
convergence theorem that the second term approaches 0 as
$r\rightarrow 0+$. Therefore, it remains to show that the first
term approaches 0 as $r\rightarrow 0+$. Since
$\supp(\overline\partial\psi_r)\subset B(A,r)\setminus
B\left(A,\frac{r}{2}\right)$ and $\supp(\phi)\subset U$,
\begin{align*}\abs{\int_{\reg{M}\cap U}
f \dbar\psi_r\wedge\phi}&\leq C \esup_{z\in \left(B(A,r)\setminus
B\left(A,\frac{r}{2}\right)\right)\cap U} \abs{f(z)}\cdot
\frac{1}{r} \cdot \mathcal{H}\left(B(A,r)\cap \reg{M}\right)\\ &\leq C
\left(\esup_{{z\in \left(B(A,r)\setminus
B\left(A,\frac{r}{2}\right)\right)\cap U}}
\abs{f(z)}\right)\frac{1}{r}r^{2n-d-1} &\text{by
Lemma~\ref{lem-volumes}} \\ &\leq C S_f(A,r)r^{2n-d-2}.
\end{align*} Our result now follows from equation
\eqref{eq-condcr}.
\end{proof}

\begin{proof}[Proof of Theorem~\ref{thm-removable}]

{\it (i)} We fix some stratification of $M$ compatible with $E$. Then for
each $p\in E$ we have $S(A(p),r)\rightarrow 0$ as $r\rightarrow 0$.
However, since $k(p)\leq 2n-2$ for each point $p\in E$, we have
$2n-k(p)-2\geq 0$. Therefore,~\eqref{eq-condcr} is satisfied, and the result follows from
Proposition~\ref{prop-cr}.

{\it (ii)} We first verify that $f\in L^1_{\rm loc}(M)$ (this is true even if $E$ has codimension one
in $M$). We fix a stratification of $M$ compatible with $E$. Let
$p\in E$, and let $A$ be the stratum of $M$ through $p$. We define
for $\nu\in \mathbb{N}$,
\[ K_\nu=\{z\in M\colon 2^{-(\nu+1)}<\dist(z,A)\leq
2^{-\nu}\}.\] By Lemma~\ref{lem-volumes},
\begin{align*}
\mathcal{H}(K_\nu)&\leq \mathcal{H}(B(A,2^{-\nu})\cap M)\\
&=  O( (2^{-\nu})^{2n-\dim A -1})\\
&=O(  2^{-\nu}). \end{align*}
Note also that if $z\in K_\nu$,
$\abs{f(z)}= O(2^{\nu\alpha})$. Therefore,
\begin{align*} \int_{\reg{M}\cap
B(A,1)}\abs{f(z)}d\mathcal{H}(z)&\leq
C\sum_{\nu=0}^\infty2^{\nu\alpha}\cdot2^{-\nu}\\
&=\frac{ C}{1-2^{\alpha-1}} \\
&< \infty.
\end{align*}

To verify that $f$ is CR , we note that for any $p\in E$, we have
$S(A(p),r)\leq Cr^{-\alpha}$ and $2n-k(p)-2\geq 1$. Therefore, near $p$,
$S(A(p),r)r^{2n-k(p)-2}\leq Cr^{1-\alpha}\rightarrow 0$ as
$r\rightarrow 0+$, and the result again follows from Proposition~\ref{prop-cr}.
\end{proof}

Let $M$ be a smooth hypersurface in $\cx^n$, minimal at a
point $p\in M$. Let $E$ be a smooth real submanifold of $\cx^n$ of
codimension two, and $E\subset M$. Then, after shrinking $M$,
$M\setminus E$ has two components $M^\pm$. Let $f$ be the function
on $M$ which is 1 on $M^+$ and 0 on $M^-$. Then $f\in L^\infty_{\rm
loc}(M)\subset L^1_{\rm loc}(M)$, and $f\in \CR(M\setminus E)$. We
claim that $f$ is {\em not} CR. Indeed, if $f$ were CR, it would
extend holomorphically to a one-sided neighbourhood of  $p$ by
Tr\'{e}preau's theorem. Then the extension had to be both
identically 1 and identically 0, which is a contradiction.
Therefore, $f$ is not CR, and {\em singularities  of codimension 1
are not in general removable for locally bounded functions.}

\subsection{CR functions as boundary values of holomorphic functions}\label{sec-bv}

Let $M$ be a smooth real hypersurface in $\mathbb C^n$. Let $f$ be a
bounded holomorphic function defined on one side of $M$ (more
generally we may assume that $f$ is of polynomial growth near $M$).
Then there is a well-defined CR distribution ${\rm bv}\, f$ on $M$
acting on a test function $\chi$ on $M$ by
$$
{\rm bv}\, f(\chi) = \lim_{t\rightarrow 0+}\int_{M} f(z+t\nu(z))\chi(z)
d\mathcal{H},
$$
where $\nu$ is the unit normal vector to $M$ with a suitable
orientation. The distribution ${\rm bv}\, f$ is called the boundary
value of $f$ on $M$. In particular, if $f$ is continuous up to $M$,
then simply ${\rm bv}\, f = f|_M$. For more details, see \cite{ber} or \cite{kytbook}.
Conversely, given a distribution $u$ on $M$, we say that $u$ extends
to a holomorphic function $f$ on one side of $M$, if ${\rm bv}\, f =
u$.

For a subanalytic hypersurface $M$ and a one-sided connected
neighbourhood $\omega$ of $M$ we say that a holomorphic function
$F$ on $\omega$ is the holomorphic extension of a CR function $f$ on
$M$ if $f|_{\reg{M}}$ is the boundary value of $F$ on the smooth
hypersurface $\reg{M}$ in the sense of the previous paragraph.

Just as for smooth hypersurfaces, continuous boundary values of
holomorphic functions on subanalytic hypersurfaces are CR. More
precisely, the following holds.

\begin{prop}\label{prop-bv} Let $M$ be a subanalytic hypersurface in $\cx^n$ and
let $\Omega$ be a domain in $\cx^n$ such that $\Omega\setminus M$
has two components $\Omega^+$ and $\Omega^-$. Let $f$ be a
continuous function on $\Omega^+\cup(M\cap\Omega)$ such that
$f|_{\Omega^+}$ is holomorphic. Then $f|_{M\cap \Omega}$ is CR.
\end{prop}

\begin{proof}
By Theorem~\ref{thm-removable}(ii), we only need to consider the
case when $\sng{M}$ has codimension one in $M$. We fix a
stratification of $M$ compatible with $\sng{M}$. Thanks to
Theorem~\ref{thm-removable}(ii) again, it is sufficient to show that
$f$ is CR near every point $p$ of any stratum $A$ of dimension
$2n-2$ contained in $\sng{M}$. Fix $p$, and shrink $\Omega$ around
$p$ so that $\Omega\cap M$ is the disjoint union of $A\cap \Omega$
and $\reg{M}\cap\Omega$ (this is possible by the frontier
condition.) To prove that $f|_{M\cap \Omega}$ is CR we need to show
that
\begin{equation}
\int_{M^{\rm reg}\cap \Omega} f \dbar\phi =0
\end{equation}
for any $\phi \in \mathcal{D}^{n,n-2}(\Omega)$.

By Sard's theorem, we can find a sequence $\epsilon_j\searrow 0$
such that the smooth hypersurfaces $A_j=\{z\in\Omega\colon
\dist(z,A)=\epsilon_j\}$ meet $\reg{M}$ transversely. We set
$A_j^+=A_j\cap\overline{\Omega^+}$,  $ M_j= \{z\in
\reg{M}\colon\dist(z,A)\geq \epsilon_j\},$ and $M_j^+=M_j\cup
A_j^+$. Note that $M_j^+$ divides $\Omega$ into two open sets (non
necessarily connected)
\[\Omega_j^+=\{z\in\Omega^+\colon\dist(z,A)>\epsilon_j\}\]
and
\[ \Omega_j^-=\Omega\setminus{\overline{\Omega_j^+}}.\]

We orient $\reg{(M_j^+)}$ by declaring that at each point the normal
into the set $\Omega_j^+$ is the positively directed normal. This
allows us to define CR functions on $M_j^+$ in the usual way. Fixing
$j$ for the time, we claim that $f|_{M_j^+}$ is CR. Indeed, since
$M_j$ is smooth, the restriction of $f$ to the interior of $M_j$ is
CR (boundary value of a holomorphic function on a smooth boundary).
Moreover, $f|_{A_j^+}$ is also CR (restriction of a holomorphic
function). Thus, it remains to show that $f$ is CR near any point
$q\in A_j\cap M_j$. For this we adapt the argument in
\cite[Proposition~4]{dipi}. It suffices to show that there exists a
small neighbourhood $\omega$ of $q$ in $\mathbb C^n$ such that for
any form $\psi\in\mathcal{D}^{n,n-2}(\omega)$,
\begin{equation}\label{eq:dp0}
\int_{\reg{(M_j^+\cap\omega)}}f\dbar \psi
=\int_{\mathrm{int}(M_j)\cap \omega} f\dbar \psi +
\int_{\mathrm{int} (A_j^+)\cap\omega}f\dbar \psi=0.
\end{equation}

Let $\omega$ be a small enough ball with centre at $q$ in $\cx^n$,
so that, by the Baouendi-Tr\'{e}ves theorem
\cite[Thm~2.4.1]{ber} the CR function $f$ can be approximated
uniformly on $\reg{M}\cap\omega$ by a sequence of holomorphic
polynomials $f_\nu$ on $\cx^n$. Further, by Sard's Theorem, after
shrinking $\omega$, we can assume that the sphere $\partial\omega$
meets $M_j$ transversely, so that Stokes' theorem applies to the
domain $\mathrm{int}(M_j)\cap\omega$. Then
\begin{equation}\label{eq:dp-left}
\int_{\mathrm{int}(M_j)\cap \omega}f \dbar \psi =
\lim_{\nu\to\infty} \int_{\mathrm{int}(M_j)\cap \omega} f_\nu \dbar \psi =
\lim_{\nu\to\infty} \int_{\mathrm{int}(M_j)\cap \omega} d(f_\nu \psi) =
\int_{M_j\cap A_j\cap \omega} f\psi,
\end{equation}
where the last equality holds by Stokes' theorem.

Now we let $\{V_k\}$ be an exhaustion of
$\mathrm{int}(A_j^+)\cap\omega$ by smoothly bounded relatively
compact subdomains. Note that $f$ is actually holomorphic in a
neighbourhood of each $V_k$.  A parallel computation gives \[
\int_{V_k}f \dbar \psi =\int_{V_k}\dbar(f\psi) =\int_{V_k}d(f\psi)
=\int_{\partial V_k}f\psi.\]  
 Letting $k$ go to
infinity, we obtain, using the continuity of $f$:
\begin{equation}\label{eq:dp3}
\int_{\mathrm{int}(A_j^+)\cap \omega} f \dbar \psi = \int_{M_j\cap
A_j\cap \omega} f\psi.
\end{equation}
Since in equations~\eqref{eq:dp-left} and \eqref{eq:dp3} the integrals on the right are taken
with the opposite orientation, \eqref{eq:dp0} follows. Hence, $f$ is CR on $M_j^+\cap\omega$, and
therefore, on $M^+_j$.

Since $f\in {\rm CR}(M^+_j)$, it follows that for any $\phi\in\mathcal{D}^{n,n-2}(\Omega)$,
$$
\int_{{\rm int}(M_j)} f\dbar\phi = -\int_{{\rm int}(A_j)} f\dbar \phi.
$$
Therefore,
$$
\int_{M^{\rm reg}\cap\Omega} f\dbar \phi = \lim_{j\to\infty} \int_{{\rm int}(M_j)} f\dbar\phi =
- \lim_{j\to\infty} \int_{{\rm int}(A_j)} f\dbar \phi  =0,
$$
since $\mathrm{Vol}(A_j)\rightarrow 0$. This proves the proposition.
\end{proof}

If in Proposition~\ref{prop-bv} we assume more about $M$, then we can relax the condition that
$f$ is continuous up to the boundary. For example, Theorem~\ref{thm-removable}(ii) implies that
if $\sng{M}$ has codimension at least two, and $f$ is a bounded holomorphic function on
$\Omega^+$,  then ${\rm{bv}} f$ is a CR function. This gives examples of non-continuous CR functions
on $M$ as boundary values of holomorphic functions on $\Omega^+$.

\subsection{Jump formula}\label{sec-jump}
We recall some facts regarding the jump representation of CR functions, which in the smooth case goes 
back to the work of Andreotti-Hill \cite{ah} and Chirka \cite{chirka:cr}  (see also \cite{kytbook} 
for a detailed account). Let $\Omega$ be a domain in $\cx^n$ such that  $H^{0,1}(\Omega)=0$ (for example 
 we can take  $\Omega$ to be pseudoconvex). Let $M$ be a subanalytic hypersurface, which is closed in $\Omega$ such that
 $\Omega\setminus M$ consists of two connected components $\Omega^\pm$. Orient 
$\reg{M}$ such that the positive normal points into $\Omega^+$ at each point. Denote by $[M]$ the current
of integration of degree one in $\Omega$ defined by $[M]\phi=\int_{\reg{M}}\phi$ for compactly supported 
smooth $(2n-1)$ forms $\phi$ in $\Omega$ (this is well defined since $M$ has locally
finite $\mathcal{H}$-measure.) Let $[M]=[M]^{0,1}+[M]^{1,0}$ be the natural splitting of $[M]$ into currents 
of bidegree $(0,1)$ and $(1,0)$ respectively.  Let $f$ be a CR function on~$M$. Then the fact that $f$ is CR (i.e., 
equation \eqref{eq-crdef} holds)  can be expressed in the language of currents by the equation $\dbar(f[M]^{0,1})=0$.

Since $H^{0,1}(\Omega)=0$,   the equation
\begin{equation}\label{eq-dbar}
\dbar u = f[M]^{0,1}
\end{equation}
can be solved for a distribution $u$ on $\Omega$. We set $f^\pm= u|_{\Omega^\pm}$. Then $f^\pm$ are
holomorphic on $\Omega^\pm$, and a study of the local behavior of $f^\pm$ near $\reg{M}$ using 
the Bochner-Martinelli trasform (see, e.g., \cite{kytbook}) shows that
 the following {\it Jump formula} holds in the sense of distributions on $\reg{M}$:
\begin{equation}\label{eq-jump}
f={\mathrm{bv}} f^+ -{\mathrm{bv}} f^-.
\end{equation}
Since $f\in L^1_{\rm loc}(M)$, we have,  in fact, a stronger result
that for every compact $K\subset \Omega$, 
\[
  \lim_{\epsilon\to 0^+} \int_{\reg{M}\cap K} \abs{f^+(\zeta +\epsilon\nu(\zeta)) -
  f^-(\zeta - \epsilon\nu(\zeta)) -f(\zeta)} d\mathcal{H}(\zeta) =0,
\]
 where $\nu(\zeta)$ is the unit normal vector to~$\reg{M}$ at
$\zeta\in M$.

\section{Non-extendable CR functions: Proof of Theorem~\ref{thm-non-ext}}
\label{sec-non-ext}

\subsection{Preliminaries}
In this section we give two sufficient conditions for the existence
of CR functions on a subanalytic hypersurface which do not admit
local holomorphic extensions to either side. The first condition is
non-minimality together with an additional topological assumption.
The second situation when non-extendable CR functions arise is due
to two-sided support.

Throughout the section $\mu$ will denote the multi-index
$\mu=(\mu_1,\ldots,\mu_n)$, where the $\mu_j$ are non-negative
integers, and $\abs{\mu}=\mu_1+\ldots+\mu_n$, further,
\[
D^\mu := \frac{\partial^{\abs{\mu}}}{\partial x_1^{\mu_1}\partial x_2^{\mu_2}\ldots \partial x_n^{\mu_n}}
\]
will denote the  operator of partial differentiation of order $\mu$.
With this notation the {\em Leibniz formula} takes the form
\begin{equation}\label{eq-leibniz}
D^\mu(fg)= \sum_{k\leq \mu}\frac{\mu!}{k!(\mu-k)!} D^kf\;D^{\mu-k}g,
\end{equation}
where $\mu!=\mu_1!\mu_2!\ldots\mu_n!$ and $k\leq \mu$ means that
$k_j\leq \mu_j$ for each $j=1,\ldots,n$.

 For an open set $\Omega\subset\rl^n$, and for an integer
$m\geq 0$, we denote as usual by $\mathcal{C}^m(\Omega)$ the space
of functions on $\Omega$ whose partial derivatives of order $\leq m$
exist and are continuous. By $\mathcal{C}^m_b(\Omega)$ we denote the
subspace of  $\mathcal{C}^m(\Omega)$ consisting of functions with
{\em bounded} partial derivatives of order $\leq m$. For $f\in
\mathcal{C}^m(\Omega)$ and $x\in \Omega$, let
\[ \abs{f(x)}_{\mathcal{C}^m}= \sum_{\abs{k}\leq m} \abs{D^k f(x)}.\]
Then, using the Leibniz rule, it is easy to prove by induction that for $|\mu|=m$,
\begin{equation}\label{eq:difff}
 \left|D^\mu \left(\frac{1}{g(x)}\right)\right|\leq
C_m \frac{\left(|g(x)|_{\mathcal{C}^{m}}\right)^m} {{\left|g(x)\right|}^{1+m}}.
\end{equation}

The following two lemmas will be used in the proof of Theorem~\ref{thm-non-ext}.

\begin{lemma}\label{lem-w}
Let $\Omega\Subset\rl^n$ be a bounded domain, and $E$ be a closed subset of $\Omega$. 
Let $m\ge 0$. Suppose that $f\in \mathcal{C}^m(\Omega\setminus E)$, and $g\in\mathcal{C}^m_b(\Omega)$ 
be such that $\abs{g(x)}\leq C\dist(x,E)$ for some $C>0$, and for every multi-index $\mu$, $|\mu|\le m$, there 
are constants $B_\mu>0$, $p(\mu)\geq 0$ such that
\[
\abs{D^\mu f(x)}\leq B_\mu\;\dist(x,E)^{-p(\mu)}.
\]
Then there exists an integer $L$ such that the function
\begin{equation}\label{eq-defh}
 h=\begin{cases} fg^{L} & \text{ on $\Omega\setminus E$}\\
0 & \text{on $E$}
\end{cases}
\end{equation} is in $\mathcal{C}^m_b(\Omega)$.
\end{lemma}

\begin{lemma}\label{lem-chi}Let $E_1, E_2$ be closed subanalytic subsets
of $\rl^n$, let $Z=E_1\cap E_2\not=\emptyset$ and let $z\in Z$. Then
there exist a neighbourhood $\Omega$ of $z$ in $\cx^n$, a function
$\chi\in\mathcal{C}^\infty(\Omega\setminus Z)$ with $0 \leq \chi\leq
1$, and a constant $r>0$ such that:
\begin{enumerate}
\item $\chi\equiv 1$   in a neighbourhood $U_1$ of $E_1\setminus Z$ in $\Omega\setminus Z$; and $\chi\equiv 0$ in a
neighbourhood $U_2$ of $E_2\setminus Z$ in~$\Omega\setminus Z$.
\item Neighbourhoods $U_1$ and $U_2$ can be chosen in such a way that there exists a constant $c$ such that if $x$ is
not in $U_1$ (resp.  $U_2$) then
\[
\dist(x,E_1)\geq c\;\dist(x,Z)^r \, \left(\text{resp.  }  \dist(x,E_2)\geq c\;\dist(x,Z)^r\right).
\]
\item For every integer $m\geq 0$, there is a constant $C>0$ such that for any $\mu$ with $\abs{\mu}=m$,
\begin{equation}\label{eq-dmuchi}
\abs{D^\mu(\chi(x))}\leq C\;\dist(x,Z)^{-rm^3}.
\end{equation}
\end{enumerate}
\end{lemma}

In order not to interrupt the flow of  the proof, we postpone the
proofs of Lemmas \ref{lem-w} and \ref{lem-chi} to
$\S$\ref{lemchiproofsection}.

\subsection{Non-Minimality}\label{sec-non-min} Recall that a
subanalytic hypersurface $M$ is said to be non-minimal at a point
$p\in M$, if there is a {\em complex} hypersurface $A$ in a
neighbourhood $\omega$ of $p\in\cx^n$ which passes through the point
$p$ and is contained in $M$. In general, as shown by
Theorem~\ref{mainthm}, unlike the smooth case, non-minimality does
not directly imply the existence of  non-extendable CR functions.
The following proposition proves Theorem~\ref{thm-non-ext} if condition (1) holds.

\begin{prop}\label{prop-non-min}
Let $M$ be a subanalytic hypersurface in $\cx^n$ and $p\in M$. Suppose there is a germ of a complex hypersurface
$Z\subset M$ such that $p\in Z$,  $Z$ divides $M$ locally into more than one component at $p$. Then, for every
integer $m\geq 0$ there is a CR function on $M$ near $p$ of class $\mathcal{C}^m$ which does not extend as a
holomorphic function to either side.
\end{prop}

\begin{proof}
Let $\Omega$ be a ball centred at $p$ such that $\Omega$ is divided by $M$ into two components $\Omega^\pm$, and such
that there exists $\phi\in \mathcal{O}(\Omega)$ with $Z\cap\Omega =\phi^{-1}(0)$. After shrinking $\Omega$, we
may assume that $\omega:= \Omega\cap M$ is divided by $Z$ into more than one component. We can therefore write
$\omega\setminus Z$ as the disjoint union of two non-empty open sets $\omega^+$ and $\omega^-$, which have $Z$ as
their common boundary. By Lemma~\ref{lem-chi}, after shrinking the ball $\Omega$ if required, there is a
$\chi\in\mathcal{C}^\infty(\Omega\setminus Z)$ such that $\chi\equiv 1$ on $\overline{\omega^-}\setminus
Z$, and $\chi\equiv 0$ on $\omega^+\setminus Z$, and for some $r>0$, we have $\abs{D^\mu\chi(z)}\leq C
\dist(z,Z)^{-r\abs{\mu}^3}$. Therefore, by Lemma~\ref{lem-w}, for a fixed $m$, there is an integer
$L\geq 0$, such that $f=\phi^L\chi\in\mathcal{C}^m(\Omega)$.

We claim that $f|_\omega$ is a CR function on $M$ near $p$, and $f$ does not extend holomorphically near $p$
to either $\Omega^+$ or $\Omega^-$. Indeed, on $\omega$ the function $f$ is given by
\[
f(z)= \begin{cases}
            0 &\text{for}\quad z\in \omega^+\cup (Z\cap\omega)\\
            \phi(z)^{L}&\text{for} \quad z\in \omega^-.
        \end{cases}
\]
Since at each point on $\omega\setminus Z$, $f$ is the restriction of a holomorphic function defined in a
neighbourhood of that point, $f$ is CR on $\omega\setminus Z$, continuous, and vanishes on $Z$. Therefore, by
Theorem~\ref{thm-removable}(i) it is CR on~$\omega$. Clearly, by the boundary uniqueness theorem, $f$ cannot extend as
a holomorphic function to either of the open sets $\Omega^\pm$.
\end{proof}

We remark here that it is well-known that if $M$ is a
$\mathcal{C}^\infty$-smooth hypersurface, then near non-minimal
points  there are CR functions of class $\mathcal{C}^\infty$ that do
not extend holomorphically to either side of $M$. Under precisely
what hypotheses there exist non-extendable
$\mathcal{C}^\infty$-smooth CR functions on (singular) subanalytic
hypersurfaces is an open question.

\subsection{Proper two-sided support}\label{sec-twosided}
We first recall a notion that was introduced in \cite{cs}.
\begin{defn}\label{def-2ss}Let $M$ be a subanalytic hypersurface
in $\Omega\subset\cx^n$, and let $p\in M$. We say that $M$ has {\sl
two-sided support} at $p$ if there are germs of complex analytic
hypersurfaces $A^\pm\subset\overline{\Omega^\pm}$ which pass through
$p$. We say that it has {\sl proper two-sided support} at $p$ if
$A^\pm$ may be taken to be different, and such that
\begin{equation}
\label{proper2sscondn} A^+\cap M= A^-\cap M.\end{equation}
\end{defn}

Note that according to this definition, non-minimality is a special
case of (non-proper) two-sided support, namely, when $A^+$ and $A^-$
coincide. However, unlike non-minimality, proper two-sided support
cannot occur at smooth points:

\begin{prop}\label{prop-nograph}
If a point $p\in M$ admits two-sided support by distinct complex hypersurfaces on the two sides,
then $M$ cannot be represented in holomorphic coordinates near $p$ as a graph over a real hyperplane.
Hence, $A^+\cap A^-\subset \sng{M}$.
\end{prop}

\begin{proof}
Suppose that $M$ is represented near $p=0$ as a graph over a real
hyperplane $H$, and let $L$ be any complex two-dimensional linear
subspace of $\cx^n$ transverse to $H$. Then, $M\cap L$ is
represented as a graph over $H\cap L$ and has proper two-sided
support by $A^\pm\cap L$. Assume without loss of generality that
$A^+$ is situated above the graph $M$ and $A^-$ below it. Let $v$ be
a vector in $L$ orthogonal to $H\cap L$. We set $B_t= \{z+tv\colon
z\in A^\pm\cap L\}$. Then $B_t$ is a complex curve in $L$, and for
$t>0$, we have
\begin{equation}\label{con1}
 B_t\cap C=\emptyset,
 \end{equation}
where $C=A^-\cap L$. On the other hand, $B_0\cap C$ contains the
point 0 and as $t\rightarrow 0+$, $B_t\rightarrow B_0$. We claim
that this situation is not possible. This can be deduced from
general properties of intersection of analytic varieties (see e.g.
\cite{chirka:cas}), but we give a simple proof. Let $U$ be a
neighbourhood of 0 in $\cx^2$, and let $f_t$ be a family of
holomorphic functions on $U$, depending continuously on $t$ such
that $B_t\cap U=f_t^{-1}(0)$. Let $\Delta$ be the unit disc in
$\cx$, and let $\phi:\overline{\Delta}\rightarrow C\cap U$ be a
Puiseux parametrization of $C$ near 0 such that $\phi(0)=0$. Let
$g_t=f_t\circ\phi$. Then $g_0(0)=0$, and $g_0$ is holomorphic on the
closed unit disc. It follows that there exists $\e>0$ such that for
$|t|<\e$,  we have
\[ \frac{1}{2\pi i} \int_{\partial\Delta}\frac{{g_t}'(\zeta)}{{g_t}(\zeta)}d\zeta=
\frac{1}{2\pi
i}\int_{\partial\Delta}\frac{g_0'(\zeta)}{g_0(\zeta)}d\zeta,\] since
the integral on the left assumes only integer values (the number of
zeros of $g_t$ in $\Delta$) and depends continuously on $t$.
Therefore, ${g_t}(z)=0$ for some $z\in\Delta$, and thus both
${B}_t$ and $C$ pass through $\phi(z)$, which contradicts~\eqref{con1}.
\end{proof}

We note that two-sided support occurs frequently in nature. In fact,
if $A^\pm$ are distinct complex hypersurfaces in an open set
$\Omega$ in $\cx^n$ such that $E=A^+\cap A^-$ is non-empty with
$p\in E$, and each of $A^\pm\setminus E$ is connected (this happens
when $A^\pm$ are irreducible), then
\[ M= \{z\in\Omega\colon \dist(z,A^+)=\dist(z,A^-)\}\]
has proper two-sided support at $p$. It is easy to verify that
$\Omega^\pm=\{z\in\Omega\colon\pm(\dist(z,A^+)-\dist(z,A^-))>0\}$
are connected, and it follows from \cite[Remarks~3.11]{bm} that $M$
is subanalytic (and real-analytic if $A^\pm$ are smooth.) If $M$ is
not minimal at $p$, after a small perturbation, we get a
hypersurface $\tilde{M}$ which has two-sided support at $p$ by
$A^\pm$, and which is minimal at $p$. In \cite{cs}, the quadratic
cones (zero-sets of real quadratic forms in $\cx^n$) with two-sided
support were classified.

We now prove the other half of Theorem~\ref{thm-non-ext}.

\begin{prop}\label{prop-2ss}
Let $M$ be a subanalytic hypersurface in $\cx^n$ and $p\in M$.
Suppose that  $M$ has proper two-sided support at $p$. Then,
for every integer $m\geq 0$ there is a CR function on $M$ near
$p$ of class $\mathcal{C}^m$ that does not extend as a
holomorphic function to either side.
\end{prop}

\begin{proof} Let $A^\pm$ be the two supports of $M$ on the opposite
sides at $p$. Let $\Omega$ be a neighbourhood of $p$ in $\cx^n$
such that $\Omega\setminus M$ has two components $\Omega^\pm$,
and $A^\pm\subset\overline{\Omega^\pm}$. After shrinking
$\Omega$, we may assume that there are holomorphic functions
$\phi_\pm$ on $\Omega$ such that $A^\pm=\phi_\pm^{-1}(0)$. We
set $Z=A^+\cap A^-$. It follows that $A^+\cap M=A^-\cap M=Z$.

We now  construct a $\mathcal{C}^m$-smooth CR function on $M$ near
$p$ which does not admit a local holomorphic extension to either of
$\Omega^\pm$.

By Lemma~\ref{lem-chi}, there exist a function $\chi\in\mathcal{C}^\infty(\Omega\setminus Z)$ such that $\chi\equiv
0$ in a neighbourhood $U^-$ of $A^-\setminus Z$, $\chi\equiv 1$ in a neighbourhood $V$ of $M\setminus Z$, and $r>0$
such that $\abs{D^\mu\chi(z)}\leq C \dist(z,Z)^{-r\abs{\mu}^3}$ for any multi-index~$\mu$. Further,
\begin{align}
\abs{\phi_-(z)} &\geq C\;\dist(z,A^-)^{s}&\text{{\L}ojasiewicz's inequality \eqref{eq:loj-ineq}}\nonumber\\
&\geq C\;\dist(z,Z)^{rs}&\text{by conclusion (2) of Lemma~\ref{lem-chi}.}
\end{align}
Define 
\[ g= \begin{cases}
\displaystyle{\frac{\chi}{\phi_-}} & \text{on\ } \Omega\setminus A^-,\\
0 &\text{on \ } A^-\setminus Z.
\end{cases}
\]
Then $g$ is smooth on $\Omega\setminus Z$ and vanishes on $U^-$. Therefore, using the Leibniz rule and
\eqref{eq:difff}, we conclude that
for any multi-index $\mu$,
\begin{align*}
\abs{D^\mu g(z)} &\leq
C \sum_{k\leq \mu}\abs{D^k \phi_-(z)^{-1}}\abs{D^{\mu-k}\chi(z)}\\
&\leq
C \sum_{k\leq \mu}\abs{\phi_-(z)^{-|k|-1}}\abs{D^{\mu-k}\chi(z)}\\
&\leq
C\dist(z,Z)^{-2|\mu|rs}\;\dist(z,Z)^{-r\abs{\mu}^3}\\
&\leq C \dist(z,Z)^{-p(\mu)},
\end{align*}
where the constants are independent of $z$ (but may depend on
$\phi_-$.) For a fixed $m$, by Lemma~\ref{lem-w}, there exists
$L\geq 0$ such that the function
\[ h^-= \begin{cases}
\phi_+^L\;g & \text{on\ } \Omega\setminus Z,\\
0 &\text{on \ } Z.
\end{cases}
\]
is $\mathcal{C}^m$-smooth on $\Omega$. By interchanging the role of
$\phi_+$ and $\phi_-$ we may construct the same way a
$\mathcal{C}^m$-smooth function $h^+$. Since the restriction of the
function $h^+-h^-$ to $M^{\rm reg}$ is CR, and the function vanishes
on $\sng{M}$, it follows from Theorem~\ref{thm-removable}(i), that
it is CR on $M$. By construction, $h^+-h^-$ cannot have holomorphic
extension to either side of $M$. Indeed, by the boundary uniqueness
theorem, such an extension must coincide with the meromorphic
function
\[ \tilde{f}
=\displaystyle{\frac{\phi_+^{L}}{\phi^-}-\frac{\phi_-^{\tilde L}}{\phi^+}}\]
on $\Omega\setminus(A^+\cup A^-)$ and therefore cannot be defined on
a one-sided neighbourhood of $p$ on either side.
\end{proof}

\subsection{Proofs of Lemma~\ref{lem-w} and Lemma~\ref{lem-chi}}\label{lemchiproofsection}

\begin{proof}[Proof of Lemma~\ref{lem-w}] Suppose first that
$f\in \mathcal{C}^m_b(\Omega\setminus E)$. Then $fg$ becomes
continuous on $\Omega$ if it is extended by 0 on $E$, which proves
the result for $m=0$. For $m=1$, we can use the definition of
partial derivatives to check that $h=fg^2$ (again extended by 0 on
$E$) is in $\mathcal{C}^1(\Omega)$. For $m\geq 2$, we may take
$h=fg^{m+1}$; the proof is an easy induction using the Leibniz
formula.

By the last paragraph, to prove the general case, it suffices
to show that for a given $f\in\mathcal{C}^m(\Omega\setminus E)$, there
is an $L$ such that $fg^L\in\mathcal{C}^m_b(\Omega\setminus E)$. Let
$q(m)= \max_{\abs{\mu}\leq m}p(\mu)$.  For every $m$, then there
exists $C$ such that $\abs{D^\mu f}\leq C\dist(\cdot,E)^{-q(m)}$,
and $q(m)$ is increasing in $m$. For a fixed $m$, we let $L$ be an
integer such that $L\geq q(m)+m$. If $\mu$ is any multi-index such
that $\abs{\mu}\leq m$, then
\begin{align*}
\abs{D^\mu(g^Lf)}&\leq C\;\sum_{k\leq \mu}\abs{D^{\mu-k}g^L}\abs{D^k f}\\
&\leq C\; \sum_{k\leq \mu} \abs{g}^{L-\abs{k}}\dist(\cdot,
E)^{-q(\abs{k})}\\
&\leq C\;\dist(\cdot,E)^{L-\abs{\mu}}\dist(\cdot,E)^{-q(\abs{\mu})}\\
&\leq C \; \dist(\cdot,E)^{L-m-q(m)}.
\end{align*}
By the choice of $L$, we have
$fg^L\in\mathcal{C}^m_b(\Omega\setminus E)$.
\end{proof}

For the proof of Lemma~\ref{lem-chi}, we will need to use the
following fact regarding the existence of a regularized distance
function in $\rl^n$:

\begin{thmb}
For any closed subset $E\subset \rl^n$, there is a $\mathcal{C}^\infty$ function $\delta$ on $\rl^n\setminus E$
such that,
\begin{enumerate}
\item $C_1\;\dist(x,E)\leq \delta(x)\leq C_2\;\dist(x,E)$, for $x\in
\rl^n\setminus E$, and
\item for every multi-index $\mu$, we have
\begin{equation}\label{dmudelta} \abs{D^\mu\delta(x)}\leq
\frac{B_\mu}{\delta(x)^{\abs{\mu}-1}},\end{equation}
\end{enumerate}
where the constants $C_1,C_2$ and $B_\mu$ are independent of $E$.
\end{thmb}

\begin{proof}[Proof of Lemma~\ref{lem-chi}]
In this proof we denote by $C$ any constant which is independent of
the point $x\in\rl^n\setminus Z$.

Let $\lambda$ be a $\mathcal{C}^\infty$-smooth function on $\rl$
with values in the interval $[0,1]$ such that
\[ \lambda(t)=\begin{cases} 1 & \text{if $t\leq \frac{1}{2}$}\\
0 &\text{if $t\geq 1$.}\end{cases}
\]

For $j=1,2$, let $\delta_j$  be a regularization of $\dist(z,E_j)$
as given by the result quoted above. We define $\chi\in
\mathcal{C}^\infty(\rl^n\setminus Z)$ by
\[ \chi(x)= \begin{cases}
\displaystyle{\lambda\left(\frac{\delta_1(x)}{\delta_2(x)}\right)}
&\text{ if $x
\not\in E_2$,}\\
0 &\text{ if $x \in E_2\setminus Z$}.\end{cases}\] Then conclusion (1) of the lemma holds for $\chi$, if we take
\[ U_2=\left\{\frac{\delta_1(x)}{\delta_2(x)}\geq 1\right\},\text{ and }
U_1=\left\{\frac{\delta_1(x)}{\delta_2(x)}\leq\frac{1}{2}\right\}.\]
Therefore, we need to consider only $x\in U$, where
\[ U=\left\{x\in\rl^n\setminus Z\colon \frac{1}{2}
<\frac{\delta_1(x)}{\delta_2(x)}<1 \right\}.\] Since $E_1$ and $E_2$ are regularly situated (see \eqref{eq:regsit}),
it follows that there exist a bounded neighbourhood $\Omega$ of $z$ in $\cx^n$ and $r>0$ such that for $x\in\Omega$,
\[
\dist(x,E_1)+\dist(x,E_2)\geq C\;\dist(x,Z)^r.
\]
After shrinking $\Omega$, we may assume that for $x\in \Omega$, we
have $\dist(x,Z)<1, \delta_1(x)<1, \delta_2(x)<1$. Thanks to the
comparability of $\delta_j$ and $\dist(x,E_j)$ we have
\[ \delta_1(x)+\delta_2(x)\geq C\; \dist(x,Z)^r.\]
If $x\not\in U_1$, then $\delta_1(x)>\frac{1}{2}\delta_2(x)$, and therefore, $\delta_1(x)\geq
C\,\dist(x,Z)^r$. By the comparability of $\delta_1$ with $\dist(\cdot, E_1)$, conclusion (2) follows. The estimate for
$x\not\in U_2$ follows exactly the same way. Consequently, if $x\in U\cap\Omega$, then for $j=1,2$,
\begin{equation}\label{eq-deltaj}\delta_j(x)\geq C
\dist(x,Z)^r.\end{equation}

For the last conclusion, note that it holds for $\mu=0$ if $C>1$. Now, for $x\in U_1\cup U_2$, the function $\chi$ is
locally constant. Therefore, we only need to estimate $D^\mu(\chi(x))$ for $x\in (U\cap\Omega)\setminus Z$.

First, for any multi-index $k$,
\begin{align}
\abs{D^k\left(\delta_2(x)^{-1}\right)}&\leq
\frac{C}{\delta_2(x)^{\abs{k}+1}}\cdot
\left(\abs{\delta_2(x)}_{\mathcal{C}^{\abs{k}}}\right)^{\abs{k}} &
\text{from \eqref{eq:difff}}\nonumber\\
&\leq \frac{C}{\delta_2(x)^{\abs{k}+1}}
\cdot\left(\frac{1}{\delta_2(x)^{\abs{k}-1}}\right)^{\abs{k}} &
\text{
from \eqref{dmudelta}}\nonumber\\
&=\frac{C}{\left(\delta_2(x)\right)^{\abs{k}^2+1}}\nonumber\\
&\leq \frac{C}{\dist(x,Z)^{r(\abs{k}^2+1)}} &\text{ from
\eqref{eq-deltaj}.}\label{eq-dinvest}
\end{align}
By the Leibniz rule, for any multi-index $\mu$,
\begin{align}\label{eq-musq}
\abs{D^\mu\left(\delta_1(x)\delta_2(x)^{-1}\right)} &\leq
C\sum_{k\leq\mu}\abs{D^{\mu-k}\delta_1(x)}\abs{D^k(\delta_2(x)^{-1})}\nonumber\\
&\leq
\sum_{k\leq\mu}\frac{1}{(\dist(x,Z))^{r(\abs{k}^2-\abs{k}+\abs{\mu})}}&\text{using
\eqref{dmudelta},\eqref{eq-deltaj},\eqref{eq-dinvest}}\nonumber\\
&\leq \frac{C}{\dist(x,Z)^{r\abs{\mu}^2}}
\end{align}
Finally, by \eqref{eq:difff},
\begin{align*}
\abs{D^\mu \chi}&\leq
C\abs{\lambda}_{\mathcal{C}^m}\abs{\delta_1\delta_2^{-1}}_{\mathcal{C}^m}^m\\
&\leq C
\left(\frac{1}{\dist(\cdot,Z)^{r\abs{\mu}^2}}\right)^{\abs{\mu}}
&\text{by\eqref{eq-musq}}\\
& =C\;\dist(\cdot,Z)^{-r\abs{\mu}^3},
\end{align*}
which completes the proof of \eqref{eq-dmuchi}.
\end{proof}

\subsection{Global non-extendable CR functions}
While in this paper we confine ourselves mainly to the question of
local extension, we make a few observations regarding global analogs
of the constructions of non-extendable CR functions, i.e., we
consider the question whether there can be a CR function on the
boundary $\partial\Omega$ of a domain $\Omega$ which does not have
holomorphic extension into a global one-sided neighbourhood of $M$
in $\Omega$ or its complement. If the ambient manifold is $\cx^n$
with $n\geq 2$, or a Stein manifold of dimension at least 2, by the
Bochner-Hartogs Theorem, such non-extendable CR functions do not
exist, as long as $\partial\Omega$ is smooth and connected; in fact,
every CR function on $\partial\Omega$ extends to all of $\Omega$. The 
analogous result continues to hold if $\Omega$ is a subanalytic domain in $\cx^n$. 
More precisely, the following holds.

\begin{prop}[Bochner-Hartogs Theorem]\label{prop-hartogs}
Let $\Omega\Subset\cx^n$, $n\geq 2$, be a bounded domain, such that $M=\partial \Omega$ 
is a connected subanalytic hypersurface. Let $f$ be a CR function on $M$, 
which is continuous on ${\reg{M}}$. Then there exists a function $F$  holomorphic
on $\Omega$ which is a holomorphic extension of $f$, i.e., $\mathrm{bv}F=f$
on $\reg{M}$. If for $k\geq 0$, the function $f$ is $\mathcal{C}^k$-smooth on
$\reg{M}$, then $F$ extends as a $\mathcal{C}^k$ function to $\reg{M}$.
\end{prop}

\begin{proof} By the Jump formula of $\S$\ref{sec-jump}, ${\mathrm{bv}}f^+-{\mathrm{bv}}f^-=f$, where
 $f^+$ is holomorphic on $\cx^n\setminus\overline{\Omega}$
 and $f^-$ on $\Omega$ such that  on $\reg{M}$, equation \eqref{eq-jump} holds 
in the sense of distributions.  Since $M=\partial \Omega$ is connected, so is
$\cx^n\setminus\overline{\Omega}$. Therefore, by Hartogs' theorem, $f^+$
extends to an entire function $\tilde{f^+}$ on $\cx^n$. We take
$F=\tilde{f^+}-f^-$. Then $F$ has distributional boundary values $f$
on $\reg{M}$, and the statement in the last sentence follows from \cite[Theorem~7.2.6 and Theorem~7.5.1]{ber}.
\end{proof}

Non-extendable CR functions cannot be
constructed on boundaries of bounded domains in Stein manifolds because Stein
manifolds do not contain compact complex hypersurfaces, and therefore
global analogs of non-minimality and two-sided support cannot occur.

We now consider an example (cf. Section~12 of \cite{hi}, pointed out to us by M.C. Shaw), where global
non-minimality and global two-sided support lead to the existence of non-extendable global CR functions.
Let $M \subset\cx\mathbb{P}^2$ be the compact connected real analytic hypersurface
\[ M =\{[z_0,z_1,z_2]\colon \abs{z_1}=\abs{z_2}\}.\]
$M$ is smooth except at the point $[1,0,0]$, and
$\cx\mathbb{P}^2\setminus M$ is the disjoint union of $\Omega^+$ and
$\Omega^-$, where
\[ \Omega^\pm=\{[z_0,z_1,z_2]\colon \pm(\abs{z_2}-\abs{z_1})>0\},\]
are ``Projective Hartogs Triangles." The domains $\Omega^\pm$ are
biholomorphic to each other and pseudoconvex.

$M$ is both (globally) non-minimal and has proper (global) two-sided support at the singular point $[1,0,0]$.
This allows us to construct non-extendable CR functions on $M$ in two different ways, each showing that the
Bochner-Hartogs theorem does not hold for $\Omega^\pm$.

Non-minimality: note that  $M$ is Levi-flat (in the sense that the smooth part
is Levi-flat). It is ``foliated" by projective lines $\{z_1= e^{i\theta} z_2\}$, $\theta\in\rl$ (although all
these ``leaves" pass through the singular point $[1:0:0]$).  Let $Z$ be the union of two of these leaves. For
definiteness assume that
$$
Z=\{z_1=z_2\}\cup \{z_1=-z_2\}=\{z_1^2-z_2^2=0\}.
$$ Then $M\setminus
Z$ consists of two components
\begin{align*}
M^+&= \{[z_0,z_1,z_2]\colon z_1 = e^{i\theta}z_2 \text{ for }
0<\theta<\pi\}\\
M^-&=\{[z_0,z_1,z_2]\colon z_1 = e^{i\theta}z_2 \text{ for }
\pi<\theta<2\pi\}. \end{align*}

Let $f$ be the function on $M$ defined by $f\equiv 1$ on $M^+$ and
$f\equiv 0$ on $M^-$. We claim that $f$ is a bounded CR function on
$M$ but $f$ does not extend to either $\Omega^+$ or $\Omega^-$. Indeed, on $\reg{M}$, the function $f$
is CR except along the complex hypersurface $Z$ which is tangent to
the Cauchy-Riemann vector fields $T^{0,1}(\reg{M})$, and therefore
is removable for $L^1_{\rm loc}$ CR functions on $\reg{M}$ (see
\cite[Proposition~1]{kyt-rea}.) It follows that $f$ is CR on $\reg{M}$. Therefore, by Theorem~\ref{thm-removable}(ii),
$f$ is CR on $M$. Clearly, $f$ does not extend holomorphically to $\Omega^+$ or $\Omega^-$.

Two-sided support: $M$ has global proper two-sided support. In fact
$\{z_2=0\}\subset \overline{\Omega^+}$ and
$\{z_1=0\}\subset\overline{\Omega^-}$. The function
\[ g(z) = \frac{z_1}{z_2}-\frac{z_2}{z_1},\]
is bounded on $M$ (each of the two terms has absolute value 1) and
is CR on $\reg{M}$. It follows from Theorem~\ref{thm-removable}(ii)
that $g$ is an ${L}^\infty$ CR function on $M$. However, as in the
proof of Proposition~\ref{prop-2ss}, $g$ cannot extend to either of
$\Omega^\pm$, since such an extension must blow up along $\{z_1=0\}$
and $\{z_2=0\}$. It follows that $g$ is a non-extendable CR
function.

Note, however, that there are no non-constant {\em continuous} CR functions on $M$. For any $\theta\in\rl$,
the restriction of any such function to the compact leaf $\{z_1=e^{i\theta}z_2\}\subset M$  is
holomorphic, and therefore constant. Since these leaves all pass
through the point $[1,0,0]$, the value of these constants are the
same for all leaves. This shows that the statement in
Proposition~\ref{prop-non-min} that we can construct CR functions of
arbitrary smoothness is purely local. On $M$ there are
non-extendable bounded CR functions defined globally, but there are no
continuous CR functions that do not extend.

\section{The hypersurface $M$}\label{sec-m}
\subsection{Definition  of $M$ and precise statement of the extension result.}\label{sec-mdef}
First we give a more precise  form of Theorem~\ref{mainthm}, part (i). For a point $z\in \cx^n$, 
$n\geq 2$, we will write the coordinates as  $z=(z_1,z_2,\tilde{z})$, where $z_1,z_2\in\cx$ and
$\tilde{z}\in \cx^{n-2}$ (where, as usual, if $n=2$, $\cx^0$ is taken to be the one-point space $\{0\}$).
Let $\ell\geq 2$ be an integer, $\ell=\infty$ or $\ell=\omega$. Let $\Omega$ be a neighbourhood 
of the origin in $\cx^n$, and let $\gamma$ be a $\mathcal{C}^\ell$-smooth subanalytic (i.e., its graph
is a subanalytic set) function on some neighbourhood of $\overline{\Omega}$ in $\cx^2$. Assume that 
$\gamma(0)\not=0$. We let
\begin{equation}
\label{eq-rhodef} \rho(z) = \Re(z_1z_2)+\abs{z_1}^2\gamma(z),
\end{equation}
and define
\begin{equation}\label{mgamma}
 M = \{z\in \Omega\colon \rho(z)=0\}.
\end{equation} 
Then $M$ is a subanalytic hypersurface in the sense of Definition~\ref{def:sub}. After making the 
linear change of variables $(z_1,z_2,\tilde{z})\mapsto (z_1,-z_2,\tilde{z})$, if necessary, we will further 
assume that $\gamma(0)>0$.

Let $\gamma_{\ell}(z)=1+x_2^{\ell}\abs{x_2}$. Then $\gamma_{\ell}(z)$ is a subanalytic function which is 
$\mathcal{C}^\ell$-smooth but not $\mathcal{C}^{\ell+1}$-smooth, and the same is true of $\rho$. Since $\nabla\rho(z)\not=0$  
for $z\in M^*=M\setminus\{z_1=z_2=0\}$, it follows from the implicit function theorem that $M^*$ is a hypersurface 
of smoothness at least $\mathcal{C}^\ell$. Representing  $M^*$ as a graph over $y_2\not=0$, we see that $M^*$ is 
 $\mathcal{C}^\ell$-smooth but not $\mathcal{C}^{\ell+1}$-smooth.  Since any biholomorphic 
map near 0 sending $M$ onto another $M$ for a different $\gamma_{\ell}$ must preserve the smoothness class of 
${M}^*$, it follows that the $M$'s are in general not biholomorphic for different $\gamma$'s.

\begin{theorem}\label{thm:M}
Let $M$ be defined as in~\eqref{mgamma}, and let $U\subset\Omega$ be a
neighbourhood of the origin in $\mathbb C^n$. Then there exists a
neighbourhood $V$ of the origin such that any bounded CR function
$f$ on $M\cap U$ extends to a holomorphic function $F$ in
$V^-=\{\rho<0\}\cap V$. Further, for $k\geq 0$, if $f$ is $\mathcal
C^k$-smooth on $M$, then $F$ extends to a $\mathcal C^k$-smooth
function on $V^-\cup (M\cap V)$, and $F|_{M\cap V} = f$.
\end{theorem}

Observe that $\{z_1=0\}\cap \Omega\subset M$, so that $M$ is non-minimal. For general $\ell$, 
$M^*=M\setminus\{z_1=z_2=0\}$ is a $\mathcal{C}^\ell$-smooth hypersurface with a quadratic singularity of codimension 3, 
in particular, the singularity is isolated if $n=2$. Indeed, for $c>0$, denote by $\rho_c$ the real quadratic form
\begin{equation}
\label{eqn-rhoc} \rho_c(z) = \Re(z_1z_2)+c\abs{z_1}^2
\end{equation}
on $\cx^n$, and let
\begin{equation}\label{mc}
 M_c =\{z\in\cx^n \colon \rho_c(z)=0\}.\end{equation}
Then the defining function $\rho$ of $M$ near 0 is of the form
$\rho=\rho_{\gamma(0)}+h$, where $\rho_{\gamma(0)}$ is as in
\eqref{eqn-rhoc}, and $h(z)= \abs{z_1}^2(\gamma(z)-\gamma(0))=
O(\abs{z}^3)$. If $n=2$, the real quadratic form $\rho_{\gamma(0)}$
is non-degenerate, and has two positive and two negative
eigenvalues. Therefore, by the Morse Lemma, there is a
$\mathcal{C}^\ell$-diffeomorphism of a neighbourhood of 0 onto
another neighbourhood of 0 in $\cx^2$ that maps $M$ onto the real
quadratic cone $M_{\gamma(0)}$. The latter is, in fact, the tangent
cone of $M$ at the origin. If $n\geq 3$, the quadratic form
$\rho_{\gamma(0)}$ is degenerate, but still there is a
$\mathcal{C}^\ell$-diffeomorphism $\Phi$  in a neighbourhood of $0$
in $\cx^n$ which maps $M$  onto the real quadratic cone
$M_1=\{\rho_1=0\}$. The map $\Phi$  can be given explicitly by
\begin{equation}\label{phi-eq}
\Phi(z) = \left(\sqrt{\gamma(z)} z_1, \frac{1}{\sqrt{\gamma(z)}}z_2, \tilde z\right),
\end{equation}
(valid also for $n=2$). Observe that $\Phi$ maps the complex hypersurface $\{z_1=0\}$, which makes $M$ non-minimal,
onto itself.

One can verify that the set $M_1\setminus\{z_1=0\}$ is connected,
and therefore, $M\setminus\{z_1=0\}$ is also connected. If
$\gamma(z)\equiv c$, then $\Phi$ is a $\cx$-linear isomorphism
between $M_c$, given by \eqref{mc}, and $M_1$. A simple computation
shows that there is a neighbourhood $U\subset\Omega$ of 0 such that
the set $U^-=\{z\in U\colon\rho(z)<0\}$ is  pseudoconvex. In fact
the Levi form has one positive eigenvalue  at each boundary point in
$(M\setminus\{z_1=0\})\cap U$ (which is therefore strongly
pseudoconvex if $n=2$), so the hypothesis of the Lewy extension
theorem holds.

The proof of Theorem~\ref{thm:M} can be outlined as follows. First
we construct explicitly a family of analytic discs attached to $M_1$
and use the Kontinuit\"{a}tssatz to prove that holomorphic functions
defined in some thin neighbourhood $\omega$ of
$M_1\setminus\{z_1=0\}$ extend analytically along any path in a
bigger one-sided neighbourhood of the origin, the size of which is
independent of $\omega$. This is done in Section~\ref{s:ext}. Then
we show in Section~\ref{s:env} that the analytic continuation from
$\omega$ does not yield multiple-valued functions. In the
terminology of \cite{mp}, this means that the complex hypersurface $\{z_1=0\}$ is
(locally) {\em $\mathcal{W}$-removable} at the origin. Finally, in
Section~\ref{s:proof} we conclude the proof by showing that every CR
function on $M\setminus\{z_1=0\}$ near the origin extends to a
holomorphic function on a one-sided neighbourhood of the origin, and
that the extension has the required boundary regularity on $M$.

\subsection{Construction of the extension.}\label{s:ext}
\begin{prop}\label{prop-pathext}
Let \[ M_1=\{z\in\cx^n\colon\rho_1(z)=0\},\] where
\[\rho_1(z)=\Re(z_1z_2)+\abs{z_1}^2,\]
and let $U$ be a neighbourhood of 0 in $\cx^n$. Let $\omega\subset
U^-:=U\cap\{\rho_1<0\}$ be a connected one-sided neighbourhood of
$$S_1=(M_1\setminus\{z_1=0\})\cap U.$$
Then there is a neighbourhood $V$ of the origin in $\cx^n$, such
that given any $p\in V^-:=V\cap\{\rho_1<0\}$, there is a path
$\tau\subset V^-$ starting in $\omega\cap V^-$ and terminating at
$p$, along which any holomorphic function in $\omega$ admits
analytic continuation.
\end{prop}

The proof of the proposition relies on the following lemma which will be also
used later.

\begin{lemma}\label{lem:discs}
Let $V=B\left(0,\frac{1}{3}\right)$, and let
$V^-=V\cap\{\rho_1<0\}$. There exists a continuous family of
analytic discs $\{D_w\}_{w\in V^-}$ in $\mathbb C^n$, with the
following properties.
\begin{enumerate}
\item $D_w\subset B(0,1)$.
\item $w\in D_w$.
\item $\partial D_w \subset S_1$.
\item Let $p_0=\left(\frac{1}{8},-\frac{1}{8},\tilde{0}\right)\in S_1\cap V$,
where $\tilde{0}\in\cx^{n-2}$.
 Then  the discs $\{D_w\}$ shrink to
$\{p_0\}$ as $V^-\ni w\to p_0$.
\end{enumerate}
\end{lemma}

\begin{proof}
Given the point $w=(w_1,w_2,\tilde{w})\in\mathbb C^n $, let
$\alpha=\frac{1}{2}(w_1-\overline{w_2})$. (We suppress the
dependence of $\alpha$ on $w$ for notational clarity.) Consider the
subset of $\cx$ given by
\begin{equation}\label{eq-sw}
\Sigma_w = \{\zeta\in\cx\colon \abs{\zeta-\alpha}^2
\leq\abs{\alpha}^2-\abs{w_1}^2\},
\end{equation}
which, depending on the right hand side, may be a closed disc, a
point, or empty. If $\rho_1(w)<0$, it is easily verified that
$\abs{w_1-\alpha}^2<\abs{\alpha}^2-\abs{w_1}^2$. Therefore, if $w\in
V^-$, the set $\Sigma_w$ contains the point $w_1$, and thus, it is a
disc of positive radius $\sqrt{\abs{\alpha}^2-\abs{w_1}^2}$. It also
follows from the definition of $\Sigma_w$ that if $w_1\not=0$, then
$0\not\in \Sigma_w$.

For $w\in V^-$, we consider the map $\phi:\Sigma_w\rightarrow \cx^n$
given by
\[ \phi_w:\zeta\mapsto
\left(\zeta,\frac{\abs{w_1}^2}{\zeta}-2\overline{\alpha},\tilde{w}\right),\]
and let $D_w =\phi_w(\Sigma_w)$. Note that this is well defined
since $w_1\not=0$. A computation shows that $\phi_w(w_1)=w$, and
$\rho_1(\phi_w(\zeta))=0$ if $\zeta\in\partial \Sigma_w$. Therefore,
$D_w$ is an analytic disc contained in $\{\rho_1<0\}$. It passes
through point $w$, and its boundary is attached to $M_1$.
Furthermore, $D_w\subset B(0,1)$. Indeed, to see this, it is
sufficient to show that $\partial D_w=\phi(\partial \Sigma_w)\subset
B(0,1)$, and then apply the maximum principle. If $\zeta\in
\partial \Sigma_w$, we have $\abs{\zeta-\alpha}^2
=\abs{\alpha}^2-\abs{w_1}^2$, or
\begin{equation}\label{eqa}
\abs{w_1}^2 = 2\Re\overline{\alpha}\zeta -\abs{\zeta}^2.
\end{equation}
Now, with $\zeta$ as above, we have,
\begin{align*}
\abs{\phi(\zeta)}&=\left(\abs{\zeta}^2+\abs{\frac{\abs{w_1}^2}{\zeta}
-2\overline{\alpha}}^2+\abs{\tilde{w}}^2\right)^\frac{1}{2}\\
&\leq\abs{\zeta}+\abs{\frac{\abs{w_1}^2}{\zeta}-2\overline{\alpha}}+\abs{\tilde{w}}\\
&\leq \abs{\zeta}
+\frac{2\Re\overline{\alpha}{\zeta}-\abs{\zeta}^2}{\abs{\zeta}}
+2\abs{\alpha}+\abs{\tilde{w}}& \text{(using \eqref{eqa})} \\
&= 2\abs{\alpha} +
2\frac{\Re\overline{\alpha}{\zeta}}{\abs{\zeta}}+\abs{\tilde{w}}\\
&\leq4\abs{\alpha}+\abs{\tilde{w}}&(\text{using Cauchy-Schwarz})\\
&\leq2(\abs{w_1}+\abs{w_2})+\abs{\tilde{w}}\\
&\leq3\abs{w}&(\text{using Cauchy-Schwarz again})\\
&<1.
\end{align*}
Also observe that $D_w\to\{p_0\}$ as $w\to p_0$, $w\in V^-$. This follows directly by
taking the limit in \eqref{eq-sw} and noting that $\Sigma_w$ shrinks to a point as
$w\rightarrow p_0$.
\end{proof}

\begin{proof}[Proof of Proposition~\ref{prop-pathext}]
Since $M_1$ is invariant under dilations, we may assume without loss
of generality that $U$ is the unit ball. Let $V$, $V^-$ and $p_0$ be as in
Lemma~\ref{lem:discs}. The continuous family of analytic discs $D_w$ constructed
in Lemma~\ref{lem:discs} can be used
to prove analytic continuation of holomorphic functions from $\omega$ to $V^-$.
Indeed, since $V^-$ is connected, and $p_0\in \partial V^-$, there is a path
$\tau:[0,1]\rightarrow\cx^2$, such that $\tau(0)=p_0$, $\tau(1)=p$ and
$\tau((0,1])\subset V^-$. There exists $\eta>0$ so small that $\tau^{-1}(\omega)$
contains the interval $(0,2\eta)$, and such that the interior of the disc
$D_{\tau(\eta)}$ is completely contained in $\omega$. Then the
restriction of $\tau$ to the interval $[\eta,1]$ is a path which
starts in $\omega$ and ends at $p$. We claim that any holomorphic
function on $\omega$ admits a holomorphic extension along this path
$\tau$. To see this observe that the discs $D_{\tau(t)}$ are attached to $S$, so after
shrinking them slightly, we obtain discs $\Delta_t\subset D_{\tau(t)}$ such that
$\partial\Delta_t\subset \omega$ for each $t$, $\eta\leq t\leq 1$. Since
$\Delta_\eta\subset\omega$, the result follows from the Kontinuit\"atssatz.
\end{proof}

\subsection{Schlichtness of the envelope of holomorphy of $\omega$.}\label{s:env}
It follows from Proposition~\ref{prop-pathext} that every holomorphic function on $\omega$
extends to a possibly multiple-valued holomorphic function on~$V^-$. The next proposition
shows that the extension is, in fact, single-valued.

\begin{prop}\label{prop-schlicht}
If in Proposition~\ref{prop-pathext}, the set $U$ is a ball, the
envelope of holomorphy $\mathcal{E}(\omega)$ of $\omega$ is
schlicht.
\end{prop}

\begin{proof}
Without loss of generality $U$ is the unit ball. We first show that
$\mathcal{E}(U^+)$ is schlicht. While it is possible to prove this
using a direct monodromy argument, we will deduce this from a
general result due to Trapani. Suppose that $\Omega\subset D$ are
domains in a Stein manifold. Define a {\em complex retraction} of
$D$ into $\Omega$ to be a homotopy $F_t:D\rightarrow D$, $0\leq
t\leq 1$, of holomorphic maps, such that (1) $F_0=\id_D$, (2)
$F_1(D)\subset\Omega$, and (3) for each $t$,
$F_t(\Omega)\subset\Omega$. We then have the following result
{\cite[Thm~1]{tr}}: {\em Let $D$ be a domain of holomorphy in a
Stein manifold, and $\Omega\subset D$. If there is a complex
retraction of $D$ into $\Omega$, then, $\Omega$ has a schlicht
envelope of holomorphy.} To apply this, we take
$D=U\setminus\{z_1=0\}$, $\Omega=U^+$, and $F_t(z) =
(z_1,(1-t)z_2,(1-t)\tilde{z})$. We claim that $F_t$ is a complex
retraction of $D$ into $\Omega$. Condition (1) is clear. Since
$F_1(D)=\{z\in \cx^n\colon 0<\abs{z_1}<1,
z_2=0,\tilde{z}=0\}\subset\Omega$, condition (2) follows. For
condition (3) we note that $F_t(D)\subset D$ for each $t$, and  if
$z\in \Omega$, we have $\rho_1(F_t(z))=
(1-t)\Re(z_1z_2)+\abs{z_1}^2= (1-t)\rho_1(z)+ t\abs{z_1}^2>0$, so
that $F_t(\Omega)\subset \Omega$. This shows that the envelope of
$U^+$ is schlicht.

Now we deduce that $\mathcal{E}(\omega)$ is schlicht. By the Lewy extension theorem, there is a one-sided neighbourhood
$\tilde{\omega}$ of the hypersurface $S_1$ (whose Levi form has one
positive eigenvalue) on the $U^-$ side to which every function in
$\mathcal{O}(U^+)$ extends holomorphically. After shrinking
$\tilde{\omega}$ we may assume that $\tilde{\omega}\subset\omega$.
Set $U^\sharp = U^+\cup S_1\cup \tilde{\omega}$. Then the
restriction map induces an isomorphism
$\mathcal{O}(U^\sharp)\cong\mathcal{O}(U^+)$. In particular,
$\mathcal{E}(U^\sharp)=\mathcal{E}(U^+)$.

Seeking a contradiction, assume now that the envelope $\pi:\mathcal{E}(\omega)\rightarrow \cx^n$ is multiple
sheeted. Then there exist a function $f\in \mathcal{O}(\omega)$, a point $p\in \cx^2$ and two paths
$\tau_1$ and $\tau_2$, starting in $\omega$ and ending in $p$, along which $f$ has holomorphic extensions
$f_1$ and $f_2$ such that $f_1(p)\ne f_2(p)$. Note that $U^-$ is a pseudoconvex domain containing $\omega$,
and therefore, $\tau_1, \tau_2 \subset U^-\subset\pi(\mathcal{E}(\omega))$. Further, without loss of
generality, we can assume that the paths $\tau_1$ and $\tau_2$ start in $\tilde{\omega}$.

Now let $U_1=U\setminus\{z_1=0\}$. Then $U_1$ is pseudoconvex. Note
that $U^\sharp\cup U^-= U_1$ and $U^\sharp\cap U^-=\tilde{\omega}$.
It is possible to solve a Cousin Problem in $U_1$ to obtain
functions $F^\sharp\in \mathcal{O}(U^\sharp)$ and
$F^-\in\mathcal{O}(U^-)$ such that $f|_{\tilde{\omega}}=
F^\sharp-F^-$.

For $j=1,2$, set $F_j=F^-+f_j$. Then $F_j$ is a holomorphic function
along the path $\tau_j$ which extends the function
$F^\sharp\in\mathcal{O}(U^\sharp)$. But then we have $F_1(p)\not= F_2(p)$ which contradicts the fact that
$\mathcal{E}(U^\sharp)=\mathcal{E}(U^+)$ is schlicht.
\end{proof}

\subsection{Proof of Theorem~\ref{thm:M}.}\label{s:proof}
First we note that Propositions~\ref{prop-pathext}
and~\ref{prop-schlicht} also hold for $M=\{\rho=0\}$, where the
function $\rho$ is as in Theorem~\ref{thm:M}. The crucial
observation is that for a small enough neighbourhood $U$ of $0$ in
$\cx^n$,   as in the proof of Proposition~\ref{prop-pathext},  we
can obtain discs $\Delta_w$ which pass through any specified point
in $U^-=\{z\in U\colon \rho<0\}$,  remain inside $U$, shrink as one
approaches certain boundary points, and whose boundaries are
contained in the one-sided neighbourhood $\omega$  of
$S=(M\setminus\{z_1=0\})\cap U$ to which every CR function on $S$
admits the Lewy holomorphic extension . To see this, let $c>0$ be
such that $\gamma>c$ on $U$. As before, set
$M_c=\{z\in\cx^n\colon\rho_c(z)= \Re(z_1z_2)+c\abs{z_1}^2=0\}$, and
let $U_c^-=\{z\in U\colon \rho_c(z)<0\}$. Then $U^-\subset U^-_c$,
and $M_c$ is biholomorphic to $M_1$ by a complex linear map. After
applying the linear biholomorphism, and a dilation, we may assume
that $U$ is the unit ball and $c=1$. For $w\in U^-$, we can clearly
choose $\Delta_w$ to be a subset of a connected component of
$\overline{D_w\cap U^-}$ (where $D_w$ is as in
Lemma~\ref{lem:discs}) such that the properties claimed are
verified. This provides the generalization of
Proposition~\ref{prop-pathext} to general $M$. It is easy to verify
that the proof of Prop.~\ref{prop-schlicht} also carries over, {\em
mutatis mutandis}, to this general case.

Suppose now that $f$ is a CR function on $S$. As was observed in the
previous paragraph, by the Lewy extension theorem,  $f$ extends to a
holomorphic function $\tilde{f}$ on a one-sided neighbourhood
$\omega\subset\{\rho<0\}$, and by the previous steps, $\tilde f$
extends further to a holomorphic function $F$ on  $V^-$. The
extension $F$ assumes the boundary values $f$ on $S$ in the same way
in which the Lewy extension assumes the value $f$ on $S$. To
complete the proof, we need to understand the behaviour of $F$ as
one approaches $\{z_1=0\}$.

First assume that $f$ is in $L^\infty$. Then, by the paragraph above,
$F$ has distributional boundary values $f$ on $S\cap V$. We need to
show that $F$ has distributional boundary values equal to $f$ on
$\reg{M}\cap V$.

Let $M^*=M\setminus\{z_1=z_2=0\}$. Then the fact that
\eqref{phi-eq} is a $\mathcal{C}^\ell$-diffeomorphism shows that
$M^*$ is a $\mathcal{C}^\ell$-smooth hypersurface (where $\gamma$ in
\eqref{eq-rhodef} is $\mathcal{C}^\ell$-smooth.) Clearly
$\reg{M}\subset M^*$. We will show that $F$ has distributional
boundary values $f$ on $M^*\cap V$.

 Let $\nu(z)$ denote the unit normal to $M^*$
directed towards $\{\rho<0\}$. For $t>0$, define a function $F_t$ on
$M^*$ by $F_t(z)= F(z+t\nu(z))$. Since $\abs{F}\leq \sup_S\abs{f}$,
we have $\abs{F_t(z)}\leq \sup_{S}\abs{f}$ for each $t>0$ and $z\in
M^*$. Moreover, thanks to \cite[Thm~3.1]{kytbook},
$F_t(z)\rightarrow f(z)$ at each Lebesgue point $z$ of $f|_S$, and
since $\{z_1=0\}$ has measure 0, we have on $M^*\cap V$,
\[ \lim_{t\rightarrow 0+} F_t= f
\quad\quad\quad\text{a.e.}\] If $\chi\in
\mathcal{C}^\infty_c(M^*\cap V)$, then by the dominated convergence
theorem,
\[
\lim_{t\rightarrow 0+}\int_{M^*\cap V} F(z+t\nu(z))\chi(z)
d\mathcal{H}=\int_{M^*\cap V} f(z)\chi(z) d\mathcal{H}.
\]
It follows that $f$ is the boundary  value of $F$ in the sense of distributions.

If $f\in\mathcal{C}^k(M)$, $k\geq 0$, then  the function $F$
obtained above extends as a $\mathcal{C}^k$-smooth function to
$S\cap V=(M\setminus\{z_1=0\})\cap V$, and has boundary values equal
to $f$. Since $f$ is continuous on $M^*\cap V$, and $F$ has
distributional boundary values $f$ on $M^*\cap V$, it follows from
\cite[Thms~7.2.6 and~7.5.1]{ber} that $F$ extends as a
$\mathcal{C}^k$-smooth function to $M^*\cap V$.

It remains to prove that $F$ extends as a $\mathcal{C}^k$ function
also to $\Sigma=\{z_1=z_2=0\}\cap V$. First assume that $k=0$. Let
$w\in \Sigma$. We need to show that as $z\rightarrow w$ through
points in $V^-$, we have that $F(z)\rightarrow f(w)$. Suppose that
for each $z\in V^-$, there exists a disc $\Delta(z)$ in
$\overline{V^-}$ passing through $z$ such that the boundary
$\partial\Delta(z)\subset {M}^*\cap V$, and such that $\Delta(z)$
shrinks to the point $w$ as $V^-\ni z\rightarrow w$. Applying the
maximum principle to the holomorphic function $F(z)-f(w)$ on the
disc $\Delta(z)$, we have
\begin{align*}
\abs{F(z)-f(w)}&\leq \sup_{\zeta\in\partial\Delta(z)}
\abs{F(\zeta)-f(w)}\\
&=\sup_{\zeta\in\partial\Delta(z)}\abs{f(\zeta)-f(w)}.
\end{align*}
Since $\Delta(z)$ shrinks to $w$ as $z\rightarrow w$, it follows
that
\[ \lim_{V^-\backepsilon z\rightarrow w}\sup_{\zeta\in\partial\Delta(z)}\abs{f(\zeta)-f(w)}=0.\]
To complete the proof, we specify the discs $\Delta(z)$. We can
take, for example,
\[ \Delta(z) =\{\zeta\in\overline{V^-}\colon \zeta_2=z_2,\tilde{\zeta}=\tilde{z}\}.\]
For small $z$, both required properties are easily verified.

Now let $k\geq 1$. By the previous paragraph, every partial
derivative of $F$ of order $k$ extends continuously to 0. It follows
that $F$ extends as a $\mathcal{C}^k$-smooth function to the origin.

\section{The hypersurface $N$}\label{sec-n}

\subsection{Definition of $N$ and precise statement}\label{sec-ndef}

We now define $N$ and state the general form of
Theorem~\ref{mainthm}(ii). Let $(z_1,z_2,z_3,\tilde z)\in
\cx^3\times\cx^{n-3}$ be the coordinates in $\cx^n$, $n\ge 3$. Let
$U$ be a neighbourhood of 0 in $\cx^n$, and let $h$ be a
$\mathcal{C}^\ell$-smooth real-valued function on a neighbourhood of
$\overline{U}$ such that $h(0)=0$, where $\ell\geq 2$. We let
\begin{equation}\label{eq-sigmadef}
\sigma(z) = \Re(z_1z_2+z_1\overline{z_3}) +\abs{z_1}^2 h(z),
\end{equation}
and
set
\begin{equation}\label{eq:N}
 N =\{z\in U \colon \sigma(z)=0\}
\end{equation}
As with $M$, we set $U^\pm=\{z\in U\colon \pm\sigma(z)>0\}$.

\begin{theorem}\label{thm:N}
Let $N$ be given as in~\eqref{eq:N}. For any neighbourhood $U$ of the origin, there exists
a neighbourhood $V\subset \cx^n$ of $0$ such that any bounded CR function $f$ on $N\cap U$
extends to a holomorphic function $F$ on $V$ with $F|_{N\cap V}=f$.
\end{theorem}

We note that the singularity of $N$ is degenerate, i.e., the real Hessian of $\sigma$ at 0 is not invertible.
In fact, the Hessian has two positive and two negative eigenvalues, and the remaining eigenvalues vanish.
Therefore, the Morse Lemma does not apply. However, there is a diffeomorphism $\Psi$ from a neighbourhood of the
origin in $\cx^n$ into $\cx^n$, which maps $N$ onto a neighbourhood $0$ in the cone
\begin{equation}\label{eq-n1}
N_1=\left\{z\in \cx^n\colon \sigma_1(z)=\Re(z_1z_2+z_1\overline{z}_3)=0 \right\}.
\end{equation}
If $\gamma(z)= 1+h(z)$, this map is given explicitly by
\begin{equation}\label{psi}
 \Psi(z)=\left( \sqrt{\gamma(z)}z_1,
\frac{z_2}{\sqrt{\gamma(z)}},\frac{\gamma(z)-1}{\sqrt{\gamma(z)}}z_1+
\frac{z_3}{\sqrt{\gamma(z)}}, \tilde z\right) .
\end{equation}

Note that $\Psi$ maps the hypersurface $\{z_1=0\}\subset N$ onto
itself. A computation shows that $N_1^{\rm{sing}} =\{z\in\cx^n\colon
z_1=0, z_2+\overline{z_3}=0\}$. It follows that
$N\setminus\{z_1=0\}$ is smooth.

The following notation will be used in the proof.
For $\delta>0$, denote by $B(\delta)=\{z\in\cx^n\colon\abs{z}<\delta\}$ the ball radius $\delta$ in $\cx^n$
centred at 0, and set $N^\pm(\delta)=B(\delta)\cap U^\pm$, where $U^\pm =\{z\in U\colon \pm\sigma(z)>0\}$
and $\sigma$ is the defining function of $N$ as in \eqref{eq-sigmadef}.  Note that $N^\pm(\delta)$
are one-sided neighbourhoods of 0 with respect to $N$. We use the notation $N_1^\pm(\delta)$ for
similarly defined one-sided neighbourhoods of $N_1$. The proof of Theorem~\ref{thm:N} follows a
similar pattern. After proving some results concerning envelopes of one-sided neighbourhoods of $0$
with respect to $N$, in Proposition~\ref{prop-envelope2} we characterize local envelope of holomorphy
of arbitrarily thin neighbourhoods of $N\setminus\{z_1=0\}$ at 0.

\subsection{Envelope of one-sided neighbourhoods of $N$}

\begin{prop}\label{hullprop2}
Given any $\epsilon>0$, there is a $\delta>0$ such that every
$g\in\mathcal{O}(N^+(\epsilon))$ extends to a single-valued holomorphic function in $B(\delta)\setminus\{z_1=0\}$.
\end{prop}

\begin{proof} We first prove the proposition for $N_1$. Let $L=\{z\in \cx^n: z_1=z_3\}$. Then
$$N_1\cap L=\{\Re(z_1z_2)+|z_1|^2=0,\ z_1=z_3\}.$$
Thus, $N_1\cap L$ (considered as a singular hypersurface in $L\cong\cx^{n-1}$) is equivalent to the
hypersurface $M_1$ of the previous section. By Lemma~\ref{lem:discs}, there exists a continuous
family of discs $\{D_w\}$ attached to $(N_1\cap L)\setminus\{z_1=0\}$ and passing through $w$,
where $w=(w_1,w_2,w_1,\tilde w)\in L$, and $\Re(w_1w_2)+|w_1|^2<0$. Further, the discs $D_w$ shrink
to a point $p_1=\left(a,-a,a,\tilde 0\right)$, as $w\to p_1$, where $a>0$ is sufficiently small.

Consider the translations $$T_\tau(z_1,z_2,z_3,\tilde z)=(z_1,z_2+\tau,z_3-\overline\tau,\tilde z).$$ Note that
$N_1$ is invariant under $T_\tau$, and $T_\tau(\{z_1=0\})=\{z_1=0\}$ for any $\tau\in \cx$.
It follows that for any $\tau\in \mathbb C$, the discs $T_\tau(D_w)$ are also attached to
$(N_1\cap L)\setminus\{z_1=0\}$.

Let $\delta=1/10$, and let $w=(w_1,w_2,w_3,\tilde w)$ be an arbitrary point in $N^-_1(\delta)$. We set
$\tau=\overline w_1 - \overline w_3$, and $w'=(w_1,w_2-\tau,w_1,\tilde w)$. Note that $w' \in L$, and
$$
\Re(w_1(w_2-\tau))+|w_1|^2=\Re(w_1w_2+w_1\overline w_3)<0.
$$
Therefore, $T_\tau(D_{w'})$ passes through $w$ and is attached to $(N_1\cap L)\setminus\{z_1=0\}$.
The disc $T_\tau(D_{w'})$ can be given explicitly by
$$
\phi_w(\zeta)=\left(\zeta,\frac{|w_1|^2}{\zeta}+w_2-\overline w_1,\zeta-\overline \tau, \tilde w \right),
$$
where $\zeta$ belongs to the set $\Sigma_{(w_1,w_2-\tau)}$ defined in \eqref{eq-sw}. Further, repeating the
calculations of Lemma~\ref{lem:discs}, for $w\in N^-(\delta)$ we have
$$
|\phi(\zeta)|<\sqrt{|\zeta|^2+\left|\frac{|w_1|^2}{\zeta}-(\overline w_1-w_2)\right|+
|\zeta|^2+|\tau|^2}<2\sqrt 2 |w|+3|w|<1.
$$
This shows that the discs are contained in the unit ball.

A computation shows that the Levi form of $N_1$  has one positive eigenvalue at each point of $N_1\setminus \{z_1=0\}$.
Therefore, by the Lewy extension theorem, every holomorphic function on $N_1^+(1)$ extends to each
point of $(N_1\cap B(1))\setminus\{z_1=0\}$, in particular, to the boundaries of the discs
constructed above. As before, by the Kontinuit\"atssatz, any function holomorphic in $N_1^+(1)$
admits analytic continuation along any path starting at $p_1$ and ending at any point $p$
in $N_1^-(\delta)$.

Now we show that the extension so obtained is single-valued. For this it is sufficient to show that
any loop in $B\left(\delta\right)\setminus\{z_1=0\}$ can be deformed into a path in
$N_1^+\left(\delta\right)$. Let $\phi:\left[0,\delta\right)\rightarrow [0,\infty)$ be a
diffeomorphism. Then the diffeomorphism $z\mapsto \phi\left(\abs{z}\right)z$ maps
$B\left(\delta\right)\setminus\{z_1=0\}$ to $\cx^n\setminus\{z_1=0\}$ and
$N_1^+\left(\delta\right)$ to $N_1^+=\{ z\in\cx^n \colon \sigma_1(z)>0\}$. It is therefore
sufficient to show that every path in $\cx^n\setminus\{z_1=0\}$ can be deformed to a path in
$N_1^+$.

For $\lambda$ real, let
\begin{equation}\label{nlambda}
N_\lambda^+ =\{ z\in\cx^n\colon \sigma_1(z)+(\lambda-1)\abs{z_1}^2>0\}.
\end{equation}
Then
\[ \bigcup_{\lambda\geq 1} N_\lambda^+ = \cx^n\setminus\{z_1=0\}.\]
Hence, if $\alpha$ is any bounded path in $\cx^n\setminus\{z_1=0\}$, then there exists
a $\mu\geq 1$ such that $\alpha\subset N_{\mu}^+$. Now as $s$ increases from $1$ to
$\sqrt{\mu}$, a map given by
\[ z\mapsto \left( {s}z_1,
\frac{1}{{s}}z_2,\left({s}-\frac{1}{{s}}\right)z_1+ \frac{1}{{s}}z_3,\tilde z\right)
\]
continuously deforms $N_\mu^+$ into $N_1^+$. (cf. equation~\eqref{psi}.) This proves the proposition
for $N_1$.

For the general case of hypersurface $N$, we shrink $U$ such that there exists $\lambda$, $0<\lambda<1$, with
the property that $h(z)\geq (\lambda-1)$ for $z\in U$. Then, $N_\lambda^+(\epsilon)\subset N^+(\epsilon)$ for every
$\epsilon>0$. The linear biholomorphism of~\eqref{psi} with $\gamma\equiv \lambda$ maps $N_\lambda^+$
onto $N_1^+$ while fixing $\{z_1=0\}$. It follows that there is $\delta>0$ such that every
holomorphic function on $N_\lambda^+(\epsilon)$ extends to $B(\delta)\setminus\{z_1=0\}$. Thus,
Prop.~\ref{hullprop2} follows by restricting $g$ to $N_\lambda^+(\epsilon)$.
\end{proof}

\subsection{Envelope of holomorphy of a neighbourhood of $N\setminus\{z_1=0\}$}

We now deduce the following consequence of Proposition~\ref{hullprop2}.

\begin{prop}\label{prop-envelope2}
Given a neighbourhood  $U$ of 0 in $\cx^n$ there is a neighbourhood $V$ of 0 in $\cx^n$ with the following
property. If $\omega\subset U$ is a neighbourhood of $(N\setminus\{z_1=0\})\cap U$, the envelope $\mathcal{E}(\omega)$
contains the set $V\setminus\{z_1=0\}$.
\end{prop}

Note that  the neighbourhood $\omega$ can be arbitrarily thin.

\begin{proof} Let $\epsilon>0$ be such that $B(\epsilon)\subset U$. Note that the open set $B(\epsilon)\setminus\{z_1=0\}$
is connected and pseudoconvex, and is divided into two connected components $N^\pm(\epsilon)$ by the smooth hypersurface
$(N\cap B(\epsilon))\setminus\{z_1=0\}$. Moreover, shrinking $\epsilon$ if required, we can assume that the Levi-form of
$(N\cap B(\epsilon))\setminus\{z_1=0\}$ has one positive and one negative eigenvalue at each point.

We set $N^\sharp=N^+(\epsilon)\cup\omega$ and $N^\flat = N^-(\epsilon)\cup\omega$. Then
$N^\sharp\cup N^\flat =B(\epsilon)\setminus\{z_1=0\}$, and $N^\sharp\cap N^\flat =\omega$. Now let
$f\in\mathcal{O}(\omega)$. Since $B(\epsilon)\setminus\{z_1=0\}$ is pseudoconvex, after solving the Cousin problem,
we can write, $f= f^+-f^-$, where $f^+\in\mathcal{O}(N^\sharp)$ and $f^-\in\mathcal{O}(N^\flat)$. By Prop.~\ref{hullprop2}
above, there is $\delta^+>0$ such that $f^+|_{N^+(\epsilon)}$  (and therefore, $f^+$) extends to a holomorphic function
$\tilde{f}^+$ on $B(\delta^+)\setminus\{z_1=0\}$.

The linear map $(z_1,z_2,z_3,\tilde z)\mapsto (-z_1,z_2,z_3,\tilde z)$ maps $N$ onto the singular hypersurface
$\tilde{N}=\{z\in U\colon \tilde{\sigma}(z)=0\}$, where
\[ \tilde{\sigma}(z) = \Re(z_1z_2+z_1\overline{z_3}) + \abs{z_1}^2 \tilde{h}(z),\]
and
$\tilde{h}(z_1,z_2,z_3,\tilde z)=-h(-z_1,z_2,z_3,\tilde z)$. Since this map sends $N^\pm(\epsilon)$ to $\tilde{N}^\mp(\epsilon)$,
by Prop.~\ref{hullprop2} again, there is $\delta^->0$ such that $f^-|_{N^-(\epsilon)}$ (and therefore, $f^-$) extends to a
holomorphic function $\tilde{f}^-$ on $B(\delta^-)\setminus \{z_1=0\}$.

We set $\delta= \min (\delta^+,\delta^-)$, and
\[ F = \tilde{f}^+- \tilde{f}^-.\]
Then $F$ is holomorphic on $B(\delta)\setminus \{z_1=0\}$, and $F|_\omega=f$. This completes the proof with $V= B(\delta)$.
\end{proof}

We now complete the proof of Theorem~\ref{thm:N}. Let $U$ be a neighbourhood of 0 in $\cx^n$ and let $f\in \CR(N\cap U)$
be a bounded CR function. By the Lewy extension theorem, $f$ extends to a neighbourhood $\omega$ of
$(N\setminus\{z_1=0\})\cap U$. Then, by Proposition~\ref{prop-envelope2}, $f$ extends to the set of the form
$V\setminus\{z_1=0\}$, where $V$ is a neighbourhood of $0$. The constructed extension remains a bounded function on the
complement of $\{z_1=0\}$, and therefore, by the removable singularity theorem, it admits holomorphic extension to a
neighbourhood of the origin. This completes the proof.


\end{document}